\documentclass[a4paper]{scrartcl}

\usepackage{graphicx}
\usepackage[utf8]{inputenc}
\usepackage[english]{babel}

\usepackage{amsmath,amssymb,enumerate,url,appendix,tabularx,stmaryrd,mathrsfs}

\usepackage{todonotes,comment,afterpage}
\usepackage[compatibility=false]{caption}
\usepackage{subcaption}
\usepackage{framed}
\usepackage{fullpage}
\usepackage{dsfont}

\usepackage{amsthm}
\usepackage{graphicx}

\theoremstyle{plain}
\newtheorem{theorem}{Theorem}[section]
\newtheorem{lemma}[theorem]{Lemma}
\newtheorem{corollary}[theorem]{Corollary}
\newtheorem{proposition}[theorem]{Proposition}

\newtheorem{assumption}[theorem]{Assumption}

\newtheorem{definition}[theorem]{Definition}

\newtheorem{remark}[theorem]{Remark}

\numberwithin{equation}{section}
\numberwithin{figure}{section}
\numberwithin{table}{section}

\newcommand{\R}{\mathbb{R}}
\newcommand{\C}{\mathbb{C}}

\newcommand{\bigO}{\mathcal{O}}

\newcommand{\abs}[1]{\left|#1\right|}
\newcommand{\norm}[1]{\left\|#1\right\|}

\newcommand{\eremk}{\hbox{}\hfill\rule{0.8ex}{0.8ex}} 

\newcommand{\colvec}[1]{\begin{pmatrix}#1\end{pmatrix}}

\newcommand{\support}{\operatorname{supp}}

\newcommand{\laplace}{\Delta}

\newcommand{\dualproduct}[2]{\left<#1,#2\right>_{\Gamma}}
\newcommand{\ltwoproduct}[2]{\left(#1,#2\right)_{\Omega^\pm}}

\newcommand{\tracejump}[1]{{\left\llbracket \gamma #1 \right \rrbracket}}
\newcommand{\normaljump}[1]{{\left\llbracket \partial_n #1 \right \rrbracket}}

\newcommand{\tracemean}[1]{\left\{\!\!\left\{ \gamma #1 \right\}\!\! \right\}}

\newcommand{\rkA}{A}
\newcommand{\rkb}{b}
\newcommand{\rkc}{c}
\newcommand{\ones}{\mathds{1}}

\newcommand{\ii}{i}

\newcommand{\uhat}{\widehat{u}}

\iftrue
\usepackage{tikz}
\usetikzlibrary{shapes}
\usepackage{pgfplots}
\usetikzlibrary{calc}
\usetikzlibrary{positioning}
\usetikzlibrary{arrows}

\usepackage{tkz-euclide}

\pgfplotsset{compat=1.15}
\newcommand{\includeTikzOrEps}[1]{\tikzexternalenable \tikzsetnextfilename{#1}  {\include{figures/#1}} \tikzexternaldisable}

\usetikzlibrary{external}
\tikzexternaldisable

\else
\newcommand{\includeTikzOrEps}[1]{\includegraphics{figures_pdf/#1}}
\fi

\newcommand{\dtn}{\operatorname{DtN}}
\newcommand{\dti}{\operatorname{DtI}}
\newcommand{\id}{\operatorname{I}}

\renewcommand{\Re}{\operatorname{Re}}

\newcommand{\uinc}{u^{\mathrm{inc}}}
\newcommand{\duinc}{\dot{u}^{\mathrm{inc}}}

\newcommand{\dd}{\partial_t^k}
\newcommand{\ddinv}{[\partial_t^k]^{-1}}

\title{On superconvergence of Runge-Kutta convolution quadrature for the wave equation}
\author{Jens Markus Melenk, Alexander Rieder }

\date{\today}

\begin{document}
\maketitle
\begin{abstract}

The semidiscretization of a sound soft scattering problem modelled by the wave equation 
is analyzed. The spatial treatment is done by integral equation methods. 
Two temporal discretizations based on Runge-Kutta convolution quadrature 
are compared: one relies on the incoming wave as input data and the other one 
is based on its temporal derivative. The convergence rate of the latter is shown
to be higher than previously established in the literature. Numerical
results indicate sharpness of the analysis.   
The proof hinges on a novel estimate on the Dirichlet-to-Impedance map for certain Helmholtz problems.
  Namely, the frequency dependence can be lowered by one power of $\abs{s}$
  (up to a logarithmic term for polygonal domains) compared to the Dirichlet-to-Neumann map.
 
\end{abstract}

\section{Introduction}
Boundary element methods have established themselves as one of the  standard methods when dealing with scattering problems, especially
  if the domain of interest is unbounded. First introduced for stationary problems,
  beginning with  the seminal works~\cite{BamH,BamH2} these methods have steadily been extended to time-dependent problems;
  see~\cite{sayas_book} for an overview. The method of convolution quadrature (CQ), introduced by Lubich in~\cite{lubich88a,lubich88b},
  is a convenient way of extending the stationary results to a time-dependent setting.

It is well-known that the convergence rate of a Runge-Kutta convolution quadrature (RK-CQ) as introduced in~\cite{lubich_ostermann_rk_cq},
is determined by bounds on the convolution symbol $K$ in the Laplace domain.
Namely, a bound of the form
$$
\norm{K(s)} \leq C \abs{s}^{\mu}
$$
leads to convergence rate ${q+1-\mu}$, as was proven in~\cite{BLM11}, see also~\cite{banjai_lubich_err_analysis_rkcq,lubich_ostermann_rk_cq} for earlier results
in this direction.
Thus one might expect that changing the symbol to $s^{-1} K(s)$ would increase the convergence order by one.

When considering discretizations of the wave equation using boundary integral methods, this is not always the case.
Instead, it has been observed that sometimes a ``superconvergence phenomenon'' appears, where the
observed convergence rate surpasses those predicted, see~\cite{sayas_semigroups,sayas_composite_scattering,rieder_diss}.

In this paper, we give a first explanation why such a phenomenon occurs in the  model problem of sound soft scattering,
i.e., the discretization of the Dirichlet-to-Neumann map. We expect that similar phenomena can also explain the
improved convergence rate for the Neumann problem or more complex scattering problems.
The proof relies on the observation that the $s^{-1}$-weighted Dirichlet-to-Neumann map
can be decomposed into a Dirichlet-to-Impedance map plus the identity operator.
For the Dirichlet-to-Impedance operator, it was observed in~\cite{banjai_dtn} that an improved
bound holds compared to the Dirichlet-to-Neumann map as long as the geometry is given by the
sphere or the half-space. It is then conjectured that a similar bound holds for smooth, convex geometries.
In this paper, provided that we restrict the Laplace parameter $s$  
  to a sector, we generalize this result to a much broader class of geometries, namely, smooth or polygonal geometries, 
without convexity assumption. This will then immediately give the stated improved bound for the
convolution quadrature scattering problem. In the case of polygons, the result holds
   in a slightly weaker form in that it contains an additional logarithmic term.

As a consequence of this observation, it may often be beneficial to select a problem formulation with an extra time derivative.
In many situations, such formulations are even the natural choice, see, e.g.,~\cite{banjai_nonlinear1,banjai_lubich_nonlinear2,BLS15}, and especially
when one works with the wave equation as a first order system as in~\cite{sayas_composite_scattering}.

Another way of looking at this phenomenon is that when using a standard formulation~
(i.e., taking $\lambda^k$ as in Proposition~\ref{prop:standard_method}
  without using a time derivative on the data), 
   then the discrete integral will exhibit a superconvergence effect. 

We point out that the present paper focuses on a semidiscretization of the problem with respect to the time variable.
For practical purposes one would also have to take into account the discretization in space using boundary elements.
We also mention that while popular, CQ is only one possibility to apply boundary integral techniques
to wave propagation problems. Notably also space-time based methods have gained popularity~\cite{gimperlein_eps_17,gimperlein_eps_18,joly_rodriguez17} in recent years.

\section{Model Problem and Notation}
\label{sect:model_problem}
We consider a sound soft scattering problem for acoustic waves. 
For a bounded Lipschitz domain $\Omega^- \subseteq \R^d$ with $\Omega^+:=\R^d\setminus \overline{\Omega^-}$, the problem reads: 
Find $u^{\mathrm{tot}}$ such that 
\begin{align*}
  \ddot{u}^{\mathrm{tot}}&=\laplace u^{\mathrm{tot}} \quad \text{in $\Omega^+$,} \quad
 u^{\mathrm{tot}}(t)|_{\Gamma}=0
                           \;\, \text{for $t>0$},\quad u^{\mathrm{tot}}(t)=\uinc(t) \quad{\text{for } t \leq 0}.
\end{align*}
Here $\uinc$ is  a given incoming wave, i.e., $\uinc$ also solves the wave equation, and we assume that for $t\leq 0$ it has not reached the scatterer yet.
The problem can be recast by decomposing the total wave into the incoming and outgoing wave, $u^{\mathrm{tot}}=\uinc + u$, where $u$ solves:
\begin{equation}
\begin{split}
  \label{eq:scattering_problem}
  \ddot{u}&=\laplace u \quad \text{in $\Omega^+$,} \quad 
u(t)|_{\Gamma} =-\uinc(t)|_{\Gamma}
            \;\, \text{for $t>0$,}\quad u(t)=0 \;\,{\text{for } t \leq 0}.
\end{split}
\end{equation}
This will be the problem we are discretizing.

For simplicity, we consider two possible cases. The bounded Lipschitz domain $\Omega^- \subseteq \R^d$ has either a smooth boundary or  $\Omega^- \subseteq \R^2$ is  a polygon. 
While we expect
that the results and techniques can be generalized to the case of piecewise smooth geometries, such an extension would lead to a much higher level
of technicality in the present paper.
Although we focus on the exterior scattering problem as our motivating model problem all of the main results also hold for the
interior Dirichlet problem.

We end the section by fixing some notation. We write $H^{m}(\Omega^\pm)$ for the usual (complex valued) Sobolev spaces on  $\Omega^+$ or $\Omega^-$.
On the interface $\Gamma:=\partial \Omega$ we also need fractional spaces $H^{s}(\Gamma)$ for $s \in [-1,1]$, 
see, e.g.,~\cite{book_mclean,adams_fournier}
for precise definitions. We also set $H^{1}_{\laplace}(\Omega^\pm):=\{u \in H^1(\Omega^\pm): \; \laplace u \in L^2(\Omega^\pm)\}$.
We write $\gamma^\pm: H^1(\Omega^\pm) \to H^{1/2}(\Gamma)$ for the exterior and interior trace operator, and
$\partial_n^{\pm}: H^{1}_{\laplace}(\Omega^\pm) \to H^{-1/2}(\Gamma)$ for the normal derivative. We note that in
both cases, we take the normal to point out of the bounded domain $\Omega^-$.
We write $\tracejump{u}:=\gamma^+ u - \gamma^- u$ and $\tracemean{u}:=\frac{1}{2}\left(\gamma^+ u + \gamma^- u\right)$ for
the trace jump and mean, and $\normaljump{u}:=\partial_n^+ u - \partial_n^- u$ for the jump of the normal derivative.

The notation $A \lesssim B$ abbreviates $A \leq C B$ with implied constants independent of critical parameters, in particular
the parameter $s$ that appears throughout this work. For (relatively) open sets $\mathcal{O}$ we introduce the $L^2(\mathcal{O})$ scalar
product $(u,v)_{L^2(\mathcal{O})}:= \int_{\mathcal{O}} u \overline{v}$. 

\subsection{Boundary Integral Methods and Convolution Quadrature}
It is well-known that scattering problems of the form presented in Section~\ref{sect:model_problem} can be solved by employing boundary integral methods,
see~\cite{sayas_book} for a detailed time domain treatment. For the frequency domain,
results can be found in most textbooks on the subject, see~\cite{sauter_schwab,book_steinbach,book_mclean,book_eps,book_hsiao_wendland}.

The use of boundary integral methods for discretizing the time domain scattering problem dates back to the
works~\cite{BamH,BamH2}, where also  important Laplace domain estimates of the form~\eqref{eq:standard_dtn_bound} were first shown.


For $s \in \C_+:=\big\{z \in \C: \,\Re(z) > 0\big\}$, we introduce the single and double layer potentials
\begin{subequations}
    \label{bem:eq:def_potentials}
    \begin{align}    
      \left(\mathrm{SLP}(s)\varphi\right)(x)&:=\int_{\Gamma}{\Phi(x-y;s) \varphi(y) \;dS(y)},  \label{bem:eq:def_slp}\\
      \left(\mathrm{DLP}(s)\psi\right)(x)&:=\int_{\Gamma}{\partial_{n(y)}\Phi(x-y;s) \psi(y) \;dS(y) }\label{bem:eq:def_dlp},
    \end{align}
  \end{subequations}
  where $\Phi$ is the fundamental solution for the operator $-\laplace + s^2$:
  \begin{align}
  \Phi(x;s)&:=\begin{cases} 
    \frac{\ii}{4} H_0^{(1)}\left(\ii s \abs{x}\right) & \text{ for $d=2$, }\\
    \frac{e^{-s \abs{x}}}{4\pi\abs{x}}, & \text{ for $d=3$. }
  \end{cases}                                        
  \end{align}
  Here $H_0^{(1)}$ denotes the first kind Hankel function of order zero,  see~\cite[Chap.~{9}]{book_mclean}.
Finally, we introduce the boundary integral operators induced by the potentials:
\begin{align}
  \label{bem_operators}
  V(s)&:=\gamma^{\pm} \mathrm{SLP}(s) \qquad \text{and }\qquad  K(s):=\tracemean{\mathrm{DLP}(s)}.
\end{align}

In practice, these operators can be realized explicitly as integrals over the boundary $\Gamma$ since 
for sufficiently smooth functions $\psi$, $\varphi$ the following equations hold:
\begin{subequations}
  \label{bem:eq:def_bem_operators_as_integrals}
  \begin{align*}    
    V(s) \varphi&= \int_{\Gamma}{\Phi( \cdot,y,s ) \varphi(y) \, d\Gamma(y)} 
    \quad \text{ and }\quad 
    K(s) \psi=\int_{\Gamma}{\partial_{n(y)} \Phi( \cdot,y,s ) \psi(y) \, d\Gamma(y)}. 
  \end{align*}
\end{subequations}

The operator we consider for discretizing~\eqref{eq:scattering_problem} is the Dirichlet-to-Neumann map.
\begin{definition}
  For $s \in \C_+$, given $g \in H^{1/2}(\Gamma)$, let $u$ solve
  \begin{align*}
    -\laplace u + s^2 u &=0 \quad \text{in } \R^d \setminus \Gamma \qquad \text{ and } \quad \gamma^\pm u=g.
  \end{align*}
  We then define the operators
  \begin{align}
    \label{eq:def_dtn_dti}
    \dtn^\pm(s) g:=\partial^\pm_n u \quad \text{and} \quad
    \dti^\pm(s) g:=\partial^\pm_n u \pm s \gamma^\pm u=\dtn^\pm g \pm s g.
  \end{align}  
\end{definition}

In practice, the following well-known proposition gives an explicit way to calculate $\dtn$.
\begin{proposition}[see, e.g., {\cite[Appendix~2]{LS09}}]
  The Dirichlet-to-Neumann map can be written as  
  \begin{align}
  \dtn^\pm(s) &= V^{-1}(s) \big(\mp \frac{1}{2} + K(s) \big).  
  \end{align}  
\end{proposition}

RK-CQ was introduced by Lubich and Ostermann in~\cite{lubich_ostermann_rk_cq}.
It provides a simple and general way of approximating convolution integrals by a high order method
and has the great advantage that only the Laplace transform of the convolution symbol is required. 
We only very briefly introduce the method and
notation.

Let $K$ be a holomorphic function in the half plane $\Re(s)> \sigma_0 >0$. Let $\mathscr{L}$ denote the Laplace transform and $\mathscr{L}^{-1}$ its inverse.
We (formally) introduce the operational calculus by defining
  \begin{align*}
    K(\partial_t)g:= \mathscr{L}^{-1}\big( K(\cdot) \mathscr{L} g \big),
  \end{align*}
  where $g \in \operatorname{dom}\left(K(\partial_t)\right)$ is such that the 
  inverse Laplace transform exists and the expression above is well defined.

  For a Runge-Kutta method given by the Butcher tableau $\rkA$, $\rkb^T,$ $\rkc$, the convolution quadrature approximation of $K(\partial_t)$  with time step size $k>0$ is given, 
    for any function $g: \R \to \R$ with $g(t)=0$ for $t \leq 0$,  by the expression
\begin{align*}
  \big[K(\dd) g\big](t)& :=\rkb^T \rkA^{-1} \sum_{j=0}^{\infty}{ W_{j} \big[g\big(t - j\,k + k \rkc_\ell - k\big)\big]_{\ell=1}^{m}} 
\quad \text{ with }\\
  K\bigg(\frac{\Delta(\zeta)}{k}\bigg) &= \sum_{n=0}^{\infty}{W_n \zeta^n},
\end{align*}
where the matrix valued function $\Delta$ is given by
  $$
  \Delta(\zeta):=\Big(\rkA + \frac{\zeta}{1-\zeta}\ones \rkb^T\Big)^{-1}.
  $$  
  
  The extension to operator valued functions $K$ is straight forward. In practice,
  we only consider evaluating $K(\dd) g$ at the discrete time steps 
  $t_j:=j\,k$.

We make the following assumptions on the Runge-Kutta method, slightly stronger than~\cite{BLM11}.
\begin{assumption}
  \label{ass:RK_method} 
  \begin{enumerate}[(i)]
  \item The Runge-Kutta method is A-stable with (classical) order $p\geq 1$ and stage order $q\leq p$.
  \item The stability function $R(z):=1+z \rkb^T(\id - z \rkA)^{-1} \ones$ satisfies $\abs{R(\ii t)}<1$ for $0\neq t \in \R$.  
  \item The Runge-Kutta coefficient matrix $\rkA$ is invertible.
  \item The method is stiffly accurate, i.e., $\rkb^T \rkA^{-1}=(0,\dots,0,1)$.
  \end{enumerate}
\end{assumption}
\begin{remark}
  Assumption~\ref{ass:RK_method} is satisfied by the Radau IIA and Lobatto IIIC methods, see~\cite{hairer_wanner_2}.
  Also note that the order conditions imply that $c_{m}=1$ for such methods.
\eremk
\end{remark}

Our analysis will employ the following result on RK-CQ 
using Laplace domain estimates:
\begin{proposition}[{\cite[Thm.~{3}]{BLM11}}]
  \label{prop:blm_cq_bound}
  Assume that $K$ is holomorphic in the half plane $\Re(s) > \sigma_0 >0$
  and that there exist $\mu_1,\mu_2 \in \R$ such that
  $K(s)$ satisfies the following bounds for  all $\delta>0$:
  \begin{align*}
    \abs{K(s)}&\leq C_{\sigma_0} \abs{s}^{\mu_1} \quad \text{ for } \Re(s) > \sigma_0 > 0, \\
    \abs{K(s)}&\leq C_{\sigma,\delta} \abs{s}^{\mu_2} \quad \text{ for } \Re(s) > \sigma > 0 \text{ with } \operatorname{Arg}(s) \in (-\pi/2+\delta,\pi/2-\delta).
  \end{align*}

  Assume that the Runge-Kutta method satisfies Assumption~\ref{ass:RK_method}. Let $r > \max\big(p+\mu_1,p,q+1\big)$ and
  $g \in C^{r}([0,T])$ satisfy $g(0)=\dot{g}(0)=\dots g^{(r-1)}(0)=0$. Then there exists $\bar{k} >0$  such that
  for $0<k<\bar{k}$,
  \begin{align*}
    \abs{K(\dd) g(t_n) - K(\partial_t) g(t_n)}&\leq C k^{\min(p,q+1-\mu_2)}\left(\abs{g^{(r)}(0)} + \int_{0}^{t_n}{\abs{g^{(r+1)}(\tau)}\,d\tau}\right). 
  \end{align*}
  The constant C depends on $t_n$, $\sigma_0$, $\bar{k}$, the constants $C_{\sigma_0}$, $C_{\delta}$, and the Runge-Kutta method. 
\end{proposition}

\section{Main results}
To simplify the notation, we introduce a symbol for the sectors in Proposition~\ref{prop:blm_cq_bound}.
Throughout this work we fix $\sigma_0>0$ and $\delta >0$ and set
  $$
  \mathscr{S}:=\big\{ s \in \C\colon \Re(s) > \sigma_0, \operatorname{Arg}(s) \in (-\pi/2+\delta, \pi/2-\delta) \big\}.
  $$

\begin{remark}
  The choice of $\sigma_0>0$ and $\delta >0$ in the definition of $\mathscr{S}$ is arbitrary, and all our estimates will hold for any
  choice, although all the constants will depend on $\sigma_0$ and $\delta$.
\eremk
\end{remark}

We are now able to state the main result of the paper. We start by stating the standard convergence result for discretizing
the Dirichlet-to-Neumann map.
\begin{proposition}[Standard method]
  \label{prop:standard_method}
  Let $g \in C^{r}([0,T],H^{1/2}(\Gamma))$ for some $r > p+2$ and $g(0)=\dot{g}(0)=\dots = g^{(r)}(0)=0$. 
  Let $\lambda:=\dtn^\pm(\partial_t) g$ be the exact normal derivative and
  $\lambda^k:=\dtn^\pm(\dd) g$ denote the standard CQ-approximation.
  Then there exist constants $C(T)>0$ and $\overline{k}>0$ such that
  the following estimate holds for $0\leq t \leq T$ and $0<k<\overline{k}$:
  \begin{align}
    \label{eq:convergence_cq_std}
    \norm{\lambda(t) - \lambda^k(t)}_{H^{-1/2}(\Gamma)}
    &\leq C(T) k^{q}\sum_{j=0}^{r}{\sup_{\tau\in(0,T)}{\norm{g(\tau)}_{H^{1/2}(\Gamma)}}}
  \end{align}
  with a constant $C(T)$ depending on the terminal time $T$, the Runge-Kutta method, $\Gamma$,
  and $\overline{k}$.
\end{proposition}
\begin{proof}
  Follows from the well-known bound 
  \begin{align}
    \norm{\dtn^\pm(s)}_{H^{1/2}(\Gamma) \to H^{-1/2}(\Gamma)}&\lesssim \frac{\abs{s}^2}{\Re(s)}
     \label{eq:standard_dtn_bound}
  \end{align}
  (see for example\cite{LS09}) and Proposition~\ref{prop:blm_cq_bound}.
\end{proof}

We will observe numerically in Section~\ref{sect:numerics} that Proposition~\ref{prop:standard_method} 
is essentially sharp. Thus, when considering the differentiated equation, one expects an increase in order by one, 
which follows directly from Proposition~\ref{prop:blm_cq_bound}.
But for the Dirichlet-to-Neumann map the increase of order is even greater, as long as one assumes slightly higher regularity
of the data.

\begin{theorem}[Method based on differentiated data]
  \label{thm:convergence_cq}
  Let $r > p+2$. Let  $g \in C^{r}\big([0,T],H^1(\Gamma)\big)$ 
satisfy $g(0)=\dot{g}(0)=\dots = g^{(r)}(0)=0$. 
  Let $\lambda:=\dtn^\pm(\partial_t) g$ be the exact normal derivative and
  $\lambda^k:=[\ddinv \dtn^\pm(\dd)] \dot{g}$ denote the CQ-approximation 
  using $\dot{g}$ as input data.
  Then, for all $\varepsilon > 0$,  there exist constants $C(T,\varepsilon)>0$ and $\overline{k}>0$ such that
  the following estimate holds for $0\leq t \leq T$ and $0<k<\overline{k}$:
  \begin{align}
    \label{eq:convergence_cq}
    \norm{\lambda(t) - \lambda^k(t)}_{H^{-1/2}(\Gamma)}
    &\leq C(T,\varepsilon) k^{\min(q+2-\varepsilon,p)}\sum_{j=0}^{r}{\sup_{\tau\in(0,T)}{\norm{g(\tau)}_{H^{1}(\Gamma)}}}.
  \end{align}
  The constant $C(T,\varepsilon)$ depends on $\varepsilon$, the terminal time $T$, the Runge-Kutta method, $\Gamma$,
    and $\overline{k}$. If $\Gamma$ is smooth, one can take $\varepsilon=0$.
\end{theorem}
\begin{proof}
  We apply Proposition~\ref{prop:blm_cq_bound}. By linearity, we can write the Dirichlet-to-Neummann operator as
  \begin{align*} 
  s^{-1} \dtn(s) &= s^{-1} \dti(s) + \id \quad \text{or, in the time domain,} \\
  \partial_t^{-1}\dtn(\partial_t)& =\partial_t^{-1} \dti(\partial_t) + \id(\partial_t).
  \end{align*}
  The second operator (in frequency domain) is independent of $s$.
  It is a simple calculation that in such cases,
  i.e., if $K(s)=T$ for all $s$, 
  the convolution weights satisfy $W_j=\delta_{j,0} T$.
  Thus, we have 
  \begin{align*}
    \id(\dd) g(t_{n+1})=\rkb^T \rkA^{-1} \left(g(t_n + k \rkc_\ell)\right)_{\ell=1}^{m}.
  \end{align*}
  Since stiff accuracy implies $\rkb^T \rkA^{-1}=(0,\dots,0,1)$ and $c_{m}=1$,
  the operator $\id$  is reproduced exactly by the CQ. 
  A  similar decomposition was already invoked in~\cite{BLM11} to explain a superconvergence phenomenon for a scalar problem.
  Combining standard estimates, e.g., \cite[Table~1]{LS09}, with
  Theorem~\ref{thm:dtn_bound} shows that the Dirichlet-to-Impedance map satisfies
  \begin{align*}
    \norm{\dti(s)}_{H^{1}(\Gamma)\to H^{-1/2}(\Gamma)}
    &\lesssim \abs{s}^{2} &\text{for $s \in \C_+,$}  \qquad\qquad\\
    \norm{\dti(s)}_{H^{1}(\Gamma)\to H^{-1/2}(\Gamma)}&\lesssim \sqrt{\log(\abs{s}+2)} \;
                                                &\text{for $s \in \mathscr{S}$. \qquad\qquad}
  \end{align*}
  By Proposition~\ref{prop:blm_cq_bound} and 
  by estimating the logarithmic term by 
   $C |s|^\varepsilon$ for arbitrary $\varepsilon>0$, we obtain \eqref{eq:convergence_cq}.
\end{proof}

While Theorem~\ref{thm:convergence_cq} is the main motivation for this paper, 
its proof is based on another result, which may be of independent interest.

\begin{theorem}
  \label{thm:dtn_bound}
  Let $s \in \mathscr{S}$. Let $\Omega \subseteq \R^d$ 
be a bounded smooth Lipschitz domain. The following estimate holds for the Dirichlet-to-Neumann map:
  \begin{align}
    \label{eq:dtn_bound}
    \norm{\dtn^\pm(s) g \pm s g}_{H^{-1/2}(\Gamma)}
    &\leq C \norm{g}_{H^1(\Gamma)} \qquad \qquad \qquad \forall g \in H^1(\Gamma).
  \end{align}
  If $\Omega \subseteq \R^2$ is a bounded Lipschitz polygon, then one has 
  \begin{align}
    \label{eq:dtn_bound_polygon}
    \norm{\dtn^\pm(s) g \pm s g}_{H^{-1/2}(\Gamma)}
    &\leq C \sqrt{ {\log(\abs{s}+2)}} \norm{g}_{H^1(\Gamma)} \quad \forall g \in H^1(\Gamma).
  \end{align}
  The constant $C$ depends only on $\Omega$ and the parameters $\sigma_0$, $\delta$ defining the sector $\mathscr{S}$.
\end{theorem}
\begin{proof}
  Due to its lengthy and technical nature, we defer the proof to Section~\ref{sect:proofs}.
  For smooth geometries it is shown as Corollary~\ref{cor:decomposition_smooth}.
    Polygonal domains are handled in Corollary \ref{cor:decomposition_for_polygon}.
\end{proof}

  \begin{remark}
The regularity requirement $g \in H^{1}(\Gamma)$ is stronger than the expected requirement
$g \in H^{1/2}(\Gamma)$. This is due to the construction of the boundary layer function 
    (see~\eqref{eq:definition_of_bl_function}).
\eremk
  \end{remark}

Since all our results hold for both the interior and exterior problem, we can also easily treat the
case of an indirect BEM formulation.
\begin{corollary}[Indirect formulation]
  Let $s\in \mathscr{S}$ and assume that $\Omega\subseteq \R^d$ is smooth.
  Then, the operator $V^{-1}(s)-2s$ satisfies the bound
  \begin{align}
    \norm{V^{-1}(s) \psi - 2s \psi}_{H^{-1/2}(\Gamma)}&\lesssim \norm{\psi}_{H^{1}(\Gamma)} \qquad \forall \psi \in H^1(\Gamma).
  \end{align}
  
  Let $g \in C^{r}([0,T],H^1(\Gamma)$ for some $r > p+2$ with $g(0)=\dot{g}(0)=\dots g^{(r)}(0)=0$. 
  Let $\varphi:=V^{-1}(\partial_t) g$ be the exact density and
  $\varphi^k:=[\ddinv V^{-1}(\dd)] \dot{g}$ be its CQ-approximation.

  Then there exist constants $C(T)>0$ and $\overline{k}>0$ such that
    the following estimate holds for $0\leq t \leq T$ and $0<k<\overline{k}$:
  \begin{align}
    \label{eq:convergence_cq_indirect}
    \norm{\varphi(t) - \varphi^k(t)}_{H^{-1/2}(\Gamma)}
    &\leq C(T) k^{\min(q+2,p)}\sum_{j=0}^{m}{\sup_{\tau\in(0,T)}{\norm{g(\tau)}_{H^{1}(\Gamma)}}}.
  \end{align}
  The constant $C(T)$ depends on $T$, 
   the Runge-Kutta method, $\Gamma$, and $\overline{k}$.
  
  If $\Omega\subseteq \R^2$ is a polygon, then 
  \begin{align}
    \!\! \norm{V^{-1}(s) \psi - 2s \psi}_{H^{-1/2}(\Gamma)}
    &\!\leq\! C(T)
      \sqrt{{\log(\abs{s}+2)}}\norm{\psi}_{H^{1}(\Gamma)} \ \forall \psi \in H^1(\Gamma) 
\end{align}
and 
\begin{align}
    \!\! \norm{\varphi(t) - \varphi^k(t)}_{H^{-1/2}(\Gamma)}
    &\leq C(T,\varepsilon)
      k^{\min(q+2-\varepsilon,p)}\sum_{j=0}^{m}{\sup_{\tau\in(0,T)}{\norm{g(\tau)}_{H^{1}(\Gamma)}}},
  \end{align}
  where $\varepsilon>0$ is arbitrary, and $C(T,\varepsilon)$ depends additionally on $\varepsilon$.
    
\end{corollary}
\begin{proof}
  We can write $V^{-1}(s)=\dtn^-(s) - \dtn^+(s)$. Thus the statements follows from Theorems~\ref{thm:convergence_cq} and \ref{thm:dtn_bound}.
\end{proof}

\section{Proofs}
\label{sect:proofs}
The proof of Theorem~\ref{thm:dtn_bound} hinges on three main observations, which require some technical work to be made rigorous:
\begin{enumerate}
\item
  \label{it:rem_on_proof_1}
  In 1d on $\R_+$, the interior Dirichlet-to-Neumann map is given by $g \mapsto s g$.
\item The existing $\dtn$-estimate's poor $s$ dependence   is mainly caused by  boundary layers.
\item Boundary layers are essentially a 1d phenomenon, so observation \ref{it:rem_on_proof_1} applies.
\end{enumerate}
\subsection{Preliminaries}
There are many ways of defining fractional order Sobolev spaces. A convenient way of working with them
  is by introducing them via the real method of interpolation. Given Banach spaces $\mathcal{X}_1 \subseteq \mathcal{X}_0$
  with continuous embedding
and parameters $\theta \in (0,1)$, $q\in [1,\infty)$, we define the interpolation norm and space as follows: 
\begin{subequations}
\label{eq:besov-norms}
\begin{align}
  \norm{u}_{[\mathcal{X}_0,\mathcal{X}_1]_{\theta,q}}
  &:=\left(\int_{t=0}^{\infty}{\Big( t^{-\theta}  \inf_{v \in \mathcal{X}_1} \big(
    \norm{u-v}_{\mathcal{X}_0} + t \norm{v}_{\mathcal{X}_1}\big)\Big)^q\frac{dt}{t}} \right)^{1/q} \\
  \big[\mathcal{X}_0,\mathcal{X}_1\big]_{\theta,q}&:=\Big\{ u\in \mathcal{X}_0:  \norm{u}_{[\mathcal{X}_0,\mathcal{X}_1]_{\theta,q}}<\infty \Big\}.
\end{align}
\end{subequations}

When working with the Helmholtz equation, it is convenient to work with $\abs{s}$-weighted norms:
\begin{definition}
  \label{def:weighted_norms}
For an open (or relatively open) set $\mathcal{O}$, parameters $s \in \C_+$ and
  $\theta \in \{0,1\}$, we define the weighted Sobolev norms
  \begin{align}
    \label{eq:weighted_norms}
    \norm{ u}_{\abs{s},\theta,\mathcal{O}}^2&:=\abs{u}_{H^{\theta}(\mathcal{O})}^2 + \abs{s}^{2\theta} \norm{u}^2_{L^2(\mathcal{O})}.
  \end{align}
  For $\theta \in (0,1)$, the corresponding norms are defined via interpolation. 
  If we want to include  boundary conditions, we write
  $\norm{ \cdot }_{\abs{s},\theta,\sim,\mathcal{O}}$ for the
  interpolation space between $L^2(\mathcal{O})$ and $H_0^1(\mathcal{O})$, but equipped with the weighted
  norms~\eqref{eq:weighted_norms}.
  
  The dual norms are defined by
  \begin{align*}
    \norm{ u}_{\abs{s},-\theta,\mathcal{O}}&:=\sup_{v \in \widetilde{H}^{\theta}(\mathcal{O})} \frac{\big|(u,v)_{L^2(\mathcal{O})}\big|}{\norm{v}_{\abs{s},\theta,\sim,\mathcal{O}}}.
  \end{align*}
\end{definition}
For the  most part we will be working with the closed surface $\partial \Omega$. There, the norms
  $\norm{\cdot}_{\abs{s},\theta,\partial \Omega}$ and $\norm{\cdot}_{\abs{s},\theta,\sim,\partial \Omega}$ coincide.
By Lemma~\ref{lemma:norm-equivalence} we also have for $\theta \in (0,1)$ and bounded domains $\mathcal{O}$ the norm equivalence 
\begin{equation}
\label{eq:norm-equivalence}
    \norm{ u}_{\abs{s},\theta,\mathcal{O}}^2 \sim \abs{u}_{H^{\theta}(\mathcal{O})}^2 + \abs{s}^{2\theta} \norm{u}^2_{L^2(\mathcal{O})}.
\end{equation}
with implied constants \emph{independent} of $|s|$. 

We start with some well-known $s$-explicit estimates for the (modified) Helmholtz equation.
\begin{lemma}[Well posedness]
  \label{lemma:well_posedness}
  Let $s \in \mathscr{S}$. The sesquilinear form 
  \begin{align*}
    a_{s}(u,v)
    &:=(\nabla u,\nabla v)_{L^2(\Omega^\pm)} + s^2 (u,v)_{L^2(\Omega^\pm)}    
  \end{align*}
  associated to $-\laplace + s^2$ is elliptic in the sense that, 
   using  $\zeta:=\frac{\overline{s}}{\abs{s}}$,  it satisfies
  \begin{align*}
    \Re\big(\zeta a_{s}(u,u)\big) \geq C \norm{u}_{\abs{s},1,\Omega^\pm}^2.
  \end{align*}  
\end{lemma}
\begin{proof}
  We  calculate:
  \begin{align*}
    \Re\big(\zeta a_{s}(u,u)\big) 
    &= \frac{\Re(\overline{s})}{\abs{s}}(\nabla u,\nabla u)_{L^2(\Omega^\pm)} + \frac{\Re(s)}{\abs{s}} \abs{s}^2 (u,u)_{L^2(\Omega^\pm)}.
  \end{align*}
  Since $\Re(s) \sim \abs{s}$ in the sector $\mathscr{S}$ this concludes the proof.
\end{proof}

\begin{lemma}[Trace estimates]
  \label{lemma:trace_estimates}
  For $\Re(s)\geq 0$ and $\abs{s}>\sigma_0$, let $u \in H^1(\Omega^\pm)$ satisfy
  $$-\laplace u + s^2 u =f  \in L^2(\Omega^\pm).$$
  Then the following estimates hold for the traces of $u$:  
  \begin{align}
    \norm{\partial_n^\pm u}_{H^{-1/2}(\Gamma)} 
    &\lesssim \abs{s}^{1/2}\norm{u}_{\abs{s},1,\Omega^\pm} + \abs{s}^{-1/2}\norm{f}_{L^2(\Omega^\pm)},
    \label{eq:normal_trace_estimate}\\
    \norm{\gamma^\pm u}_{H^{-1/2}(\Gamma)} &\lesssim  \abs{s}^{-1/2}\norm{u}_{\abs{s},1,\Omega^\pm}, \\
    \norm{\partial_n^\pm u \pm s \gamma^\pm u}_{H^{-1/2}(\Gamma)} &\lesssim \abs{s}^{1/2}\norm{u}_{\abs{s},1,\Omega^\pm} + \abs{s}^{-1/2}\norm{f}_{L^2(\Omega^\pm)}.
  \end{align}
\end{lemma}
\begin{proof}
  We start with the normal derivative. 
For
  any $\xi \in H^{1/2}(\Gamma)$ and any $v$ with $\gamma^\pm v=\xi$ we calculate:
  \begin{align*}
    \left| \dualproduct{\partial_n^\pm u}{\xi}\right| 
    &=\left| \ltwoproduct{\nabla u}{\nabla v} + s^2 \ltwoproduct{u}{v} - \ltwoproduct{f}{v} \right| \\
    &\lesssim\! \big(\norm{u}_{\abs{s},1,\Omega^\pm} \!+ \abs{s}^{-1} \norm{f}_{L^2(\Omega^\pm)}\!\big) \norm{v}_{\abs{s},1,\Omega^\pm} , 
  \end{align*}
  where in the last step we select $v$
  as the minimal energy function, satisfying
    \begin{align}
\label{eq:energy-minizing-extension} 
      -\laplace v + |s|^2 v =0 \;\text{in $\Omega^\pm$}, \qquad \gamma v = \xi \;\text{ on $\Gamma$}.
    \end{align}
    By \cite[Prop.~{2.5.1}]{sayas_book}, $v$
    admits the estimate $\norm{v}_{\abs{s},1,\Omega^\pm}\lesssim \abs{s}^{1/2} \|\xi\|_{H^{1/2}(\Gamma)},$ and
    \eqref{eq:normal_trace_estimate} follows.

 For the Dirichlet trace, we get using the multiplicative trace estimate
 and the same lifting $v$:
 \begin{align*}
   \left| \dualproduct{\gamma^\pm u}{\xi}\right| 
   &\!\leq\! \norm{\gamma^\pm u}_{L^2(\Gamma)} \! \norm{\xi}_{L^2(\Gamma)}
     \!\lesssim\! \norm{u}_{L^2(\Omega^\pm)}^{1/2} \norm{u}_{H^1(\Omega^\pm)}^{1/2} \norm{v}_{L^2(\Omega^\pm)}^{1/2} \norm{v}_{H^1(\Omega^\pm)}^{1/2} \\
   &\leq\abs{s}^{-1} \norm{u}_{\abs{s},1,\Omega^\pm} \norm{v}_{\abs{s},1,\Omega^\pm}
     \lesssim \abs{s}^{-1/2} \norm{u}_{\abs{s},1,\Omega^\pm} \norm{\xi}_{H^{1/2}(\Gamma)}.
 \end{align*}
 The estimate for the impedance trace then follows trivially.
\end{proof}

The previous lemma shows that for \textsl{a priori} estimates in terms of standard Sobolev norms  
the constants involved  have some $s$ dependence. The next lemmas show that 
the use of the weighted norms introduced in 
Definition~\ref{def:weighted_norms} avoids such dependencies:  

\begin{lemma}
\label{lemma:bounded-inverse}
  The operators $\gamma^\pm: H^{1}(\Omega^\pm) \to H^{1/2}(\Gamma)$ satisfy the bounds:
  \begin{align}
    \norm{\gamma^\pm u}_{\abs{s},1/2,\Gamma}
    &\lesssim \norm{u}_{\abs{s},1,\Omega^\pm}.
      \label{eq:est_dirichlet_trace_and_lifting}
  \end{align}
\end{lemma}
\begin{proof}
  The multiplicative trace estimate and Young's inequality give
  \begin{align*}
    \abs{s}^{1/2}\norm{\gamma^\pm u}_{L^2(\Gamma)}
    &\lesssim \big(\norm{u}_{H^1(\Omega^\pm)} \abs{s} \norm{u}_{L^2(\Omega^\pm)}\big)^{1/2} \\
      &\lesssim \big( \norm{u}^2_{H^1(\Omega^\pm)} + \abs{s}^2 \norm{u}^2_{L^2(\Omega^\pm)}\big)^{1/2}.
  \end{align*}
  Combining this with the standard trace estimate concludes the proof in view of (\ref{eq:norm-equivalence}).
\end{proof}

\begin{lemma}[Dirichlet problem]
  \label{lemma:lifting_and_apriori}
  Let $g \in H^{1/2}(\Gamma)$, $f \in L^2(\Omega^\pm)$. For any $s \in \mathscr{S}$ there exists a unique solution to the problem
  \begin{align*}
    -\laplace u + s^2 u &= f  \text{ in $\Omega^\pm$} \qquad \text{and}\qquad \gamma^\pm u =g.
  \end{align*}
  The function satisfies the \textsl{a priori} bound
  \begin{align}
    \label{eq:lifting_apriori}
    \norm{u}_{\abs{s},1,\Omega^\pm}&\lesssim \abs{s}^{-1} \norm{f}_{L^2(\Omega^\pm)} + \norm{g}_{\abs{s},1/2,\Gamma}.
  \end{align}
The implied constant depends only on $\Omega^{\pm}$ and the constants $\sigma_0$, $\delta$ characterizing $\mathscr{S}$. 
\end{lemma}
\begin{proof}
  Existence follows using the usual theory of elliptic problems. For the \textsl{a priori} bound, we
  first note that by \cite[Lemma~{4.22}]{melenk_sauter11}, there exists a lifting $u_D$ satisfying
  $$
-\laplace u_D + s^2 u_D =0, \qquad \text{}\qquad \norm{u_D}_{\abs{s},1,\Omega^\pm}\lesssim \norm{g}_{\abs{s},1/2,\Gamma} 
\qquad \text{and}\qquad \gamma^\pm u_D=g.
$$
  Thus the remainder $\widetilde{u}:=u-u_D$ solves:
  \begin{align*}
    -\laplace \widetilde{u} + s^2 \widetilde{u}&= f, 
\qquad \widetilde{u}|_{\Gamma} =0.
  \end{align*}
  As the sesquilinear form $a_s$ from Lemma~\ref{lemma:well_posedness} is elliptic, we get
with $\zeta$ defined there
  \begin{align*}
    \norm{\widetilde{u}}_{\abs{s},1,\Omega^\pm}^2
    &\lesssim \Re\big(\zeta a_{s}(\widetilde{u},\widetilde{u})\big) = \Re\big(\zeta\ltwoproduct{f}{\widetilde{u}}\big)  \\
    &  \leq \abs{s}^{-1} \norm{f}_{L^2(\Omega^\pm)}\Big(\!\abs{s}\norm{\widetilde{u}}_{L^2(\Omega^\pm)} \!\Big) 
    \leq \abs{s}^{-1} \norm{f}_{L^2(\Omega^\pm)} \norm{\widetilde{u}}_{\abs{s},1,\Omega^\pm}. \qedhere
  \end{align*}
\end{proof}

\begin{lemma}[Neumann problem]
  \label{lemma:lifting_and_apriori_neumann}
  Let $h \in H^{-1/2}(\Gamma)$. Then for every $s \in \mathscr{S}$ there exists a unique solution to the problem
  \begin{align*}
    -\laplace u + s^2 u &= f \text{ in $\Omega^\pm$} \qquad \text{and}\qquad \partial_n^\pm u =h.
  \end{align*}
   $u$ satisfies the \textsl{a priori} bound
  \begin{align}
    \label{eq:lifting_apriori_neumann}
    \norm{u}_{\abs{s},1,\Omega^\pm}&\lesssim  \norm{h}_{\abs{s},-1/2,\Gamma} + \abs{s}^{-1} \norm{f}_{L^2(\Omega^\pm)}.
  \end{align}  
The implied constant depends only on $\Omega^{\pm}$ and on $\sigma_0$, $\delta$ characterizing $\mathscr{S}$. 
\end{lemma}
\begin{proof}
  Follows easily from the weak formulation and~\eqref{eq:est_dirichlet_trace_and_lifting}. 
\end{proof}

  We also have the following trace inequality in a weighted $H^{-1/2}$-norm: 
\begin{lemma}
  If $-\laplace u + s^2 u =0$ we can estimate:
  \begin{align*}
    \norm{\partial_n^- u}_{\abs{s},-1/2,\Gamma} &\lesssim \norm{u}_{\abs{s},1,\Omega^-}.
  \end{align*}
\end{lemma}
\begin{proof}
  Follows easily from the weak definition of $\partial_n^- u $, 
the Cauchy-Schwarz inequality, and~\eqref{eq:est_dirichlet_trace_and_lifting}.
\end{proof}

\subsection{Smooth geometries}
\label{sec:smooth-gemoetries}
In order to prove a first version of Theorem~\ref{thm:dtn_bound}, we consider a simplified setting of smooth geometry and
Dirichlet trace. Closely following the ideas from~\cite{MS99,melenk_book}, we construct a lowest order boundary layer function that will
be the basis for all further estimates.

\begin{lemma}[Boundary fitted coordinates]
  \label{lemma:boundary_fitted_coordinates}
  Let $T: \mathcal{O} \subseteq \R^{d-1} \to \Gamma$ be a smooth local parametrization of $\Gamma$.
  Define $F: \mathcal{O} \times (-\varepsilon, \varepsilon) \to \R^d$ as
  \begin{align}
    \label{eq:def_diffeomorphism}
    F(\widehat{x},\rho):=-\rho n(\widehat{x})+ T(\widehat{x}),
  \end{align}
  where $n(\widehat{x})$ is the outer normal vector to $\Omega^-$ at the point $T(\widehat{x})$.
  
  For $\varepsilon>0$ sufficiently small, $F$ is a smooth diffeomorphism onto $F\big(\mathcal{O}\times (-\varepsilon,\varepsilon)\big)$.
  It holds that $F(\mathcal{O} \times (0,\varepsilon)) \subseteq \Omega^-$ and $F(\mathcal{O} \times (-\varepsilon,0)) \subseteq \Omega^+$. Additionally,
  $F$ satisfies
  \begin{align}
    \label{eq:transformation_matrix_structure}
    DF^{-1}(\widehat{x},\rho) DF^{-T}(\widehat{x},\rho)
    &=\begin{pmatrix}
      \widetilde{T}(\widehat{x}) & 0 \\
      0 & 1
    \end{pmatrix}
          + \rho \widetilde{R}(\widehat{x},\rho),          
  \end{align}
  where $\widetilde{T}$ and $\widetilde{R}$ are smooth and $\widetilde{T}(\widehat x)$ is invertible.
\end{lemma}
\begin{proof}  
  We only show~\eqref{eq:transformation_matrix_structure}.
  We select a smooth orthogonal basis of the tangent space at $T(\widehat{x})$, 
  denoted by $e_1(\widehat{x}),\dots, e_{d-1}(\widehat{x})$.
  This implies that 
  $Q:=\big(e_1(\widehat{x}), \dots, e_{d-1}(\widehat{x}), n(\widehat{x})\big)$ is orthogonal.
\begin{align*}
    DF(\widehat{x},\rho) & = (D_{\widehat{x}}, D_{\rho}) F(\widehat{x},\rho) = 
 (D_{\widehat{x}} T(\widehat{x}) - \rho D_{\widehat{x}}  n(\widehat{x}) , - n(\widehat{x})) \\
& 
    =Q\begin{pmatrix}
      \widetilde{T}_{1} & 0 \\
      0 & - 1
    \end{pmatrix}
          -\rho \big( D_{\widehat{x}} n(\widehat{x}),0\big).
\end{align*}
  Here $\widetilde{T}_1:=(e_1,\dots,e_{d-1})^T D_{\widehat{x}}T(\widehat{x})$, 
  and thus $\|{\widetilde{T}_1}\|_2\leq\norm{D_{\widehat{x}}T(\widehat{x})}_{2}$.
  We further compute:
   \begin{align}
\nonumber 
   DF^{-1} DF^{-T}
    &=\big(DF^T DF\big)^{-1}
      = \left(\begin{pmatrix}
        \widetilde{T}_1^T & 0 \\
        0 & -1
      \end{pmatrix}
            Q^T
            Q
            \begin{pmatrix}
              \widetilde{T}_1 & 0 \\
              0 & -1
            \end{pmatrix}
                  + \rho R_1(\widehat{x},\widehat{\rho})
                  \right)^{-1} \\
    &=\left(
      \begin{pmatrix}
        \widetilde{T}_1^T\widetilde{T}_1 & 0 \\
        0 & 1
      \end{pmatrix}
        + \rho R_1(\widehat{x},\widehat{\rho})\right)^{-1}
\label{eq:foo}
  \end{align}
  where $R_1$ collects the remaining terms. For sufficiently small $\rho>0$, 
depending only on $\norm{D_{\widehat{x}}T}_{2}$
  and $\norm{D_{\widehat{x}}n}_{2}$, we can linearize the inverse in \eqref{eq:foo}
to get~\eqref{eq:transformation_matrix_structure}
with $\widetilde{T}:=\big(\widetilde{T}_1\widetilde{T}_1^T\big)^{-1}$
(the latter inverse exists since $D_{\widehat{x}} T$ and thus also $\widetilde{T}_1$ has full rank).
\end{proof}

\begin{lemma}
  \label{lemma:asympt_expansion}
  Assume that $\Omega^-$ has a smooth boundary $\Gamma$. For any $s \in \mathscr{S}$
  and for every $u \in H^1(\Omega^-)$
  solving 
  $$
  -\laplace u + s^2 u =0
  $$
  together with $\gamma^- u \in H^2(\Gamma)$
  there exists a function
  $u_{BL} \in H^1(\Omega^-)$ with the following properties:
  \begin{enumerate}[(i)]
  \item $\gamma^- u_{BL} = \gamma^- u$,
  \item $\partial_n^- u_{BL} - s \gamma^- u_{BL} =0$,
    \label{it:asymptotic_exp_before_smoothing_DtI}
  \item $-\laplace u_{BL} + s^2 u_{BL} = f$
    with
    \begin{align}
      \label{eq:apriori_defect}
      \norm{f}_{L^2(\Omega^-)} \lesssim \abs{s}^{1/2}\norm{\gamma^- u}_{H^1(\Gamma)} + \abs{s}^{-1/2} \norm{\gamma^- u}_{H^2(\Gamma)}.
    \end{align}
  The implied constant depends only on $\Omega^-$ and $\sigma_0$, $\delta$ characterizing $\mathscr{S}$.
  \item
    \label{it:aysmptotic_exp_away_from_bdry_smooth}
    For $\varepsilon > 0$ define the set 
    $\Omega^-_\varepsilon:=\{ x \in \Omega^-: \operatorname{dist}(x,\Gamma) > \varepsilon \}$.
    Then, the following estimates hold for all $\ell \in \R$ with constants independent of $s$:
    \begin{align*}
      \norm{u_{BL}}_{H^{2}(\Omega^-_{\varepsilon})} &\leq C_{\varepsilon,\ell} \abs{s}^{-\ell}\norm{\gamma u}_{H^2(\Gamma)}.
    \end{align*}    
  \item The analogous statement also holds for the exterior problem $\Omega^+$, replacing $-s$ by $s$ in~(\ref{it:asymptotic_exp_before_smoothing_DtI}).
  \end{enumerate}
\end{lemma}
\begin{proof}
  We only show the case of the interior problem and abbreviate $g:=\gamma^- u$. 
  We work in boundary fitted coordinates $(\widehat{x},\rho)$ as described in Lemma~\ref{lemma:boundary_fitted_coordinates}.
  First assume, that $\support(g) \subset T(\mathcal{O})$, i.e., $g$ is supported by the part of the boundary parametrized by $T$.
  The change of variables formula shows that if $u$ solves $-\laplace u + s^2 u=f$, then
  $\widehat{u}:=u\circ F$ solves:
  \begin{align*}
    - \nabla \cdot \left( J  DF^{-1} DF^{-T} \nabla \widehat{u}\right)+ J s^2 \widehat{u} &= \widehat{f}  J
  \end{align*}
  with $J:={\operatorname{det}(DF)}$ and $\widehat f = f \circ F$ (see, e.g., \cite[Step~7 of proof of Thm.~4, Sec.~{6.3.2}]{evans98}). 
On the other hand, if $\widehat{u}_{BL}$ satisfies
  \begin{align*}
    - \nabla \cdot \left( J DF^{-1} DF^{-T} \nabla \widehat{u}_{BL}\right)+ J s^2 \widehat{u}_{BL} 
&= \widehat{f}_{BL},
  \end{align*}
  then $u_{BL}:=\widehat{u}_{BL} \circ F^{-1}$ solves
  $$
  -\laplace u_{BL} + s^2 u_{BL} = f_{BL}, \qquad \text{with } f_{BL}:=J^{-1}\widehat{f}_{BL} \circ F^{-1} .
  $$
  We set $\widehat{A}:=DF^{-1} DF^{-T}$ and define with $\widehat g:= g \circ T$ the function 
    \begin{align}
    \label{eq:definition_of_bl_function}
    \widehat{u}_{BL}(\widehat{x},\rho):=e^{-s \rho } \; \widehat g(\widehat{x})
  \end{align}
  in the boundary fitted coordinates.
By \eqref{eq:transformation_matrix_structure}, we have $\widehat A_{d,d} = 1 + \rho \widetilde{R}_{d,d}$. 
Differentiating out we obtain for some smooth functions $c_{ij}$, $a_i$, $b_i$, $d_i$, and $b$
\begin{align*}
& 
-\nabla \cdot (J \widehat A \nabla \widehat{u}_{BL}) + J s^2 \widehat{u}_{BL}  \\
& \qquad = - \sum_{i,j=1}^{d-1} c_{ij} \partial_{\widehat x_i} \partial_{\widehat x_j} \widehat{u}_{BL}   
-\sum_{i=1}^{d-1} a_{i} \partial_{\widehat x_i} \partial_{\rho} \widehat{u}_{BL}   
- \sum_{i=1}^{d-1} b_i \partial_{\widehat x_i} \widehat{u}_{BL}   \\
& \qquad \qquad \mbox{} 
-b \partial_{\rho} \widehat{u}_{BL} - J (1 + \rho \widetilde{R}_{d,d}) \partial^2_{\rho} \widehat{u}_{BL} + J s^2 \widehat{u}_{BL} 
\\
& = - \sum_{i,j=1}^{d-1} c_{ij} \partial_{\widehat x_i} \partial_{\widehat x_j} \widehat{u}_{BL}   
-\sum_{i=1}^{d-1} a_{i} \partial_{\widehat x_i} \partial_{\rho} \widehat{u}_{BL}   
- \sum_{i=1}^{d-1} b_i \partial_{\widehat x_i} \widehat{u}_{BL}   \\
& \qquad \qquad \mbox{} 
-b \partial_{\rho} \widehat{u}_{BL} - J \rho \widetilde{R}_{d,d} \partial^2_{\rho} \widehat{u}_{BL} 
=: \widehat f_{BL}, 
\end{align*}
where, in the last step, we exploited the definition of $\widehat{u}_{BL}$. 
  {}From its definition, one can easily see that $\widehat{u}_{BL}$ satisfies the estimates
  \begin{align*}
    \norm{\partial_\rho \widehat{u}_{BL}}_{L^2(\mathcal{O}\times \R_+)}\! +\! \norm{\partial_\rho \nabla_{\widehat{x}} \widehat{u}_{BL}}_{L^2(\mathcal{O}\times \R_+)}
     \!+\! \norm{\rho \partial_\rho^2 \widehat{u}_{BL}}_{L^2(\mathcal{O}\times \R_+)}
    &\lesssim \frac{\abs{s} \norm{\widehat{g}}_{H^1(\mathcal{O})}}{\sqrt{\Re(s)}}, \\ 
    \norm{\nabla_{\widehat{x}} \widehat{u}_{BL}}_{L^2(\mathcal{O}\times \R_+)}
    + \sum_{i,j=1}^{d-1}{\norm{\partial_{\widehat{x}_i} \partial_{\widehat{x}_j} \widehat{u}_{BL}}_{L^2(\mathcal{O}\times \R_+)}} 
   &\lesssim \frac{\norm{\widehat{g}}_{H^2(\mathcal{O})} }{\sqrt{\Re(s)}}.
  \end{align*}
  
  Transforming back gives~\eqref{eq:apriori_defect} for the part of $\Omega^-$ parametrized by $F$.
  Assertion (\ref{it:aysmptotic_exp_away_from_bdry_smooth}) follows easily from the definition, as the
  exponential decay dominates all powers of $\abs{s}$.
  This allows us to smoothly cut off $u_{BL}$ for large $\rho$ and extend it by $0$ to the whole domain.  
  For general $g$, we use a smooth partition of unity to decompose $g$ into functions with local support.
\end{proof}

As the next step, we lower the regularity requirement on $\gamma^- u$.
\begin{corollary}
  \label{cor:asymptotic_exp_after_smoothing}
  Let $\Omega^-$ have a smooth boundary $\Gamma$. For any $s \in \mathscr{S}$ and 
  for every $u \in H^1(\Omega^-)$ with $\gamma^- u \in H^{1}(\Gamma)$
  solving 
$$
-\laplace u + s^2 u =0
$$ 
  there exists a function
  $u_{BL} \in H^2(\Omega^-)$ with the following properties:
  \begin{enumerate}[(i)]
  \item
    \label{it:asymptotic_exp_after_smoothing_DtI}
    $\partial_n^- u_{BL} - s \gamma^- u_{BL} =0$.
  \item 
    \label{it:asymptotic_exp_after_smoothing_DtI_approx}
    $
      \norm{\partial_n^- (u-u_{BL}) - s\big(\gamma^- u - \gamma^- u_{BL}\big)}_{H^{-1/2}(\Gamma)}
      \lesssim \norm{\gamma^- u}_{H^{1}(\Gamma)}.
    $
  \item
    \label{it:asymptotic_exp_after_smoothing_norm}
    $\norm{u-u_{BL}}_{\abs{s},1,\Omega^-} \lesssim \abs{s}^{-1/2}\norm{\gamma^- u}_{H^{1}(\Gamma)}.$
The implied constant depends only on $\Omega^-$ and the constants $\sigma_0$, $\delta$ characterizing $\mathscr{S}$. 
  \item
    \label{it:aysmptotic_exp_away_from_bdry}
    For $\varepsilon > 0$ introduce 
    $\Omega^-_\varepsilon:=\{ x \in \Omega^-\colon \operatorname{dist}(x,\Gamma) > \varepsilon \}$.
    Then, the following estimates hold for all $\ell \in \R$ with constants independent of $s$:
    \begin{align*}
      \norm{u_{BL}}_{H^{2}(\Omega^-_{\varepsilon})} &\leq C_{\varepsilon,\ell} \abs{s}^{-\ell}\norm{\gamma^- u}_{H^1(\Gamma)}.
    \end{align*}    
\item 
  The analogous statement also holds in the case of the exterior problem upon 
  replacing $s$ by $-s$ in~(\ref{it:asymptotic_exp_after_smoothing_DtI})
  and~(\ref{it:asymptotic_exp_after_smoothing_DtI_approx}).
  \end{enumerate}
\end{corollary}
\begin{proof}
  In order to apply Lemma~\ref{lemma:asympt_expansion}, we need $H^2$-regularity of $g:=\gamma^- u$.  
  We fix a function $\widetilde{g} \in H^2(\Gamma)$ with the following properties:
  \begin{align}
    \norm{g-\widetilde{g}}_{\abs{s},{1/2},\Gamma}\leq \abs{s}^{-1/2} \norm{g}_{H^{1}(\Gamma)}
    \quad \text{and} \quad
    \norm{\widetilde{g}}_{H^2(\Gamma)}\lesssim \abs{s}^{1} \norm{g}_{H^{1}(\Gamma)}.
    \label{eq:est_mollified_gs}
  \end{align}
This can be either seen by realizing $H^1(\Gamma)$ as the interpolation space
between $L^2(\Gamma)$ and $H^2(\Gamma)$ and using \cite[Lemma]{bramble-scott78} or 
  constructed directly via the usual mollifiers as done in 
  \cite[Thm.~{2.29}]{adams_fournier}: The approximation
    estimate follows from \cite[Eqn (20)]{adams_fournier}
    and an interpolation argument. See also \cite[Sec.~{7.48}]{adams_fournier} for
    how to trade Sobolev regularity for approximation properties of the mollified function.  

  Let $\widetilde{u}$ denote the solution to
  \begin{align*}
    -\laplace \widetilde{u} + s^2 \widetilde{u} &= 0 \qquad \text{and} \qquad \gamma^- \widetilde{u}=\widetilde{g}.
  \end{align*}
  Since $\widetilde{g} \in H^2(\Gamma)$, we can apply Lemma~\ref{lemma:asympt_expansion} to construct $u_{BL}$.
  Assertion (\ref{it:asymptotic_exp_after_smoothing_DtI}) then follows by construction.
  For~(\ref{it:asymptotic_exp_after_smoothing_norm}):
  we note that by Lemmas~\ref{lemma:lifting_and_apriori} and~\ref{lemma:asympt_expansion}:
  \begin{align*}
    \norm{u-u_{BL}}_{\abs{s},1,\Omega^-}
    &\lesssim \norm{u-\widetilde{u}}_{\abs{s},1,\Omega^-} +\norm{\widetilde{u} - u_{BL}}_{\abs{s},1,\Omega^-} \\
    &\lesssim \norm{g-\widetilde{g}}_{\abs{s},1/2,\Gamma} +
      \abs{s}^{-1}\left(\abs{s}^{1/2}\norm{\widetilde{g}}_{H^1(\Gamma)} + \abs{s}^{-1/2}\norm{\widetilde{g}}_{H^2(\Gamma)}\right) \\
    &\lesssim \abs{s}^{-1/2} \norm{g}_{H^1(\Gamma)}.
  \end{align*}
  For~(\ref{it:asymptotic_exp_after_smoothing_DtI_approx}),
  we use Lemma~\ref{lemma:trace_estimates} and \eqref{eq:lifting_apriori} to get that
  \begin{align*}
    \norm{\partial_n^- (u-\widetilde{u}) - s\big(\gamma^- u - \gamma^- \widetilde{u}\big)}_{H^{-1/2}(\Gamma)}
    &\lesssim \abs{s}^{1/2}  \norm{g -\widetilde{g}}_{\abs{s},1/2,\Gamma} \lesssim \norm{g}_{H^{1}(\Gamma)}.
  \end{align*}
  Similarly, we have
  \begin{align*}
    & \norm{\partial_n^- (\widetilde{u}-u_{BL}) - s\big(\gamma^- \widetilde{u} - \gamma^- u_{BL}\big)}_{H^{-1/2}(\Gamma)} \\
    &\quad \lesssim \abs{s}^{1/2} \norm{\widetilde{u} - \widetilde{u}_{BL}}_{\abs{s},1,\Omega^-}+
      \abs{s}^{-1/2}\left(\abs{s}^{1/2}\norm{\widetilde{g}}_{H^1(\Omega^-)} + \abs{s}^{-1/2}\norm{\widetilde{g}}_{H^2(\Omega^-)}\right)  \\
    &\quad \lesssim \norm{g}_{H^{1}(\Gamma)}.
  \end{align*}
  Assertion (\ref{it:aysmptotic_exp_away_from_bdry}) follows directly from Lemma~\ref{lemma:asympt_expansion}~(\ref{it:aysmptotic_exp_away_from_bdry_smooth})
  and~\eqref{eq:est_mollified_gs}.
\end{proof}

\begin{corollary}
\label{cor:decomposition_smooth}
Let $\Omega^-\subset \R^d$ be smooth and $s \in \mathscr{S}$. 
Let $g \in H^1(\Gamma)$ and $u$ solve 
$$
-\Delta u + s^2 u = 0 \quad \mbox{ in $\Omega^-$}, 
\qquad \gamma^- u = g \quad \mbox{ on $\Gamma$.} 
$$
Then 
\begin{align}
  \label{eq:decomposition_smooth_dtn}
\|\partial_n u - s u\|_{H^{-1/2}(\Gamma)}& \lesssim \|g\|_{H^1(\Gamma)}.   
\end{align}
The analogous statement holds for the exterior problem upon replacing $s$ by $-s$ in~\eqref{eq:decomposition_smooth_dtn}.
\end{corollary}
\begin{proof}
  Follows by writing $u=u_{BL} + (u-u_{BL})$. The impedance trace of $u_{BL}$ vanishes
  by Corollary~\ref{cor:asymptotic_exp_after_smoothing} (\ref{it:asymptotic_exp_after_smoothing_DtI}).
  The impedance trace of the remainder is uniformly bounded with respect to $s$
  via Corollary~\ref{cor:asymptotic_exp_after_smoothing} (\ref{it:asymptotic_exp_after_smoothing_DtI_approx}). 
\end{proof}

\subsection{Polygons}
In this section, we consider a polygonal domain $\Omega\subset {\mathbb R}^2$ as an example of a non-smooth domain. 
In order to match the boundary layer solutions from Lemma~\ref{lemma:asympt_expansion} at corners,
we solve an appropriate transmission problem, similarly to what was done in~\cite{melenk_book}.
We refer to Figure~\ref{fig:corner_geometry} for the geometric situation.  

We first need one additional Sobolev space. For a smooth curve $\Gamma'$ and $\theta \in [0,1]$,
we introduce 
$$
\widetilde{H}^{\theta}(\Gamma')
:=\big\{u \in H^{\theta}(\Gamma'):\; \norm{u}_{\widetilde{H}^\theta(\Gamma')}:=\norm{u}_{H^{\theta}(\Gamma')} + \|d_{\partial \Gamma'}^{-\theta} u\|_{L^2(\Gamma')} < \infty\big\},
$$
where $d_{\partial \Gamma'}$ denotes the distance to the endpoints of $\Gamma'$.

\subsubsection{A transmission problem in a cone}
In this section, we investigate certain transmission problems. These will allow us
to match different boundary layer functions in the vicinity of a corner of the
domain.
We start by investigating the special case of a  transmission problem on a
sector or an infinite cone. Due to its special structure, we can derive
sharper estimates for the normal derivative than what can be obtained 
from the energy methods used in Lemma~\ref{lemma:transmission_lifting} below.
  \newcommand{\shat}{\hat{s}}

We introduce some notation.
   Given $\omega \in (0,\pi)$, we define the infinite cone
   \begin{equation}
\label{eq:cone}
   \mathcal{C}:=\{ (r\cos \varphi,r\sin\varphi)\colon r>0, \abs{\varphi} < \omega\}
   \end{equation}
   with opening angle $2 \omega$ and $\mathcal{C}^\prime$ by removing from $\mathcal{C}$ its bisector:  
\begin{equation}
\label{eq:cone-without-Gamma0}
\mathcal{C}^\prime:= \mathcal{C} \setminus \{(r,0) \colon r > 0\}. 
\end{equation}
   Next, we define the sector
   $S_{\omega}:=\{(r\cos \varphi,r\sin\varphi), \, r\in (0,1), \abs{\varphi} \in (0,\omega)\}$,
   which is just the truncated cone $\mathcal{C} \cap B_{1}(0)$.
   For its boundary, we write $\Gamma_{\pm \omega}:=\big\{ (r \cos(\pm \omega),r \sin(\pm \omega)), \; r\in (0,1)\big\}$
   for the two parts of the boundary of the sector that are adjacent 
   to the origin and set $\Gamma_{S}:= \Gamma_\omega \cup \Gamma_{-\omega}$.
Finally, we need to define the normal jump across interfaces.
If $\Gamma'$ denotes a smooth interface separating domains $\mathcal{O}_1$ and $\mathcal{O}_2$
we define the normal jump across $\Gamma'$ via
\begin{align*}
  \normaljump{u}:=\nabla u|_{\mathcal{O}_1} \cdot n_1+ \nabla u|_{\mathcal{O}_2} \cdot n_2,  
\end{align*}
where the normal vectors $n_{j}$ are taken to point out of $\mathcal{O}_j$ respectively.

  \begin{lemma}
    \label{lemma:transmission_on_cone_dimensionless} 
    Consider the solution $\uhat \in H^1(\mathcal{C})$ to the following problem
    on the infinite cone
    for $\mu>0$ and $\widehat{s} \in \mathscr{S}$ with $\abs{\widehat{s}}=1$:
    \begin{align}
      \label{eq:transmission_on_cone_dimensionless}
    -\laplace \uhat + \shat^2\uhat =0 \text{ in } \mathcal{C}',
      \quad \normaljump{\uhat}=e^{- \shat \mu x_1} \; \text{on $\R_+ \times\{0\}$}, \quad
    \uhat =0 \;\text{on $\partial \mathcal{C}$}.
    \end{align}    

    Then, the following statements hold for $\uhat$:
    \begin{enumerate}[(i)]
    \item
          \label{item:lemma:transmission_on_cone_dimensionless-i} 
         For each $\ell \in {\mathbb N}$ there exist constants 
         $C_{\ell}$, $\alpha_{\ell} >0$ such that for all $r \ge  1$
\begin{align*}
\|\uhat\|_{W^{\ell,\infty}(\mathcal{C}^{\prime}\setminus B_r(0))}
  &\leq C_{\ell} e^{-\alpha_{\ell} r}. 
\end{align*}%
\item \label{item:lemma:transmission_on_cone_dimensionless-ii} 
There exists a constant $C>0$ such that
         $\partial_n u$ satisfies the estimates
         \begin{align}
           \label{eq:estimates_on_cone_dimensionless}
           \norm{\partial_n \uhat}_{L^2(\partial \mathcal{C})}
           + \norm{ \partial_n \uhat}_{L^1(\partial \mathcal{C})}
           &\leq C. 
         \end{align}
       \end{enumerate}
       The constants  depend only on the opening angle $2\omega$, the parameter $\mu$, 
       and the choices of $\sigma_0$ and $\delta$ in the definition of $\mathscr{S}$.
     \end{lemma}
  \begin{proof}
We first show 
         (\ref{item:lemma:transmission_on_cone_dimensionless-ii})
in Steps~1--3 and then 
         (\ref{item:lemma:transmission_on_cone_dimensionless-i}) in Step~4. 

\emph{Step~1:}
    We start with an energy estimate in exponentially weighted spaces,   
    namely, for any $0<\alpha < \mu \Re(\widehat{s})$ the following estimate holds:
    \begin{equation}
\label{eq:transmission-scaled-10}
    \norm{e^{\alpha r}\nabla \uhat}^2_{L^2(\mathcal{C})} + \norm{e^{\alpha r}\uhat}^2_{L^2(\mathcal{C})}
    \leq C
    \end{equation}
    with a constant $C$ only depending on $\alpha$, $\mu$, $\omega$, and $\widehat{s}$.
    
    We fix some notation. We write
    $\norm{ u }^2_{1,\alpha}:=\norm{e^{\alpha r}\nabla u}^2_{L^2(\mathcal{C})} + \norm{e^{\alpha r}u}^2_{L^2(\mathcal{C})}$,
    and analogously for $\norm{ u }_{1,-\alpha}$. Also we set  $h(x_1,x_2):=e^{-\widehat{s} \mu x_1}$ for
    the transmission data.
    The proof follows  \cite[Prop.~{6.4.6}]{melenk_book} verbatim. 
    The sesquilinear form 
    $B(u,v):=(\nabla u,\nabla u)_{L^2(\mathcal{C})}+ \shat^2 (u,v)_{L^2(\mathcal{C})}$ 
satisfies an inf-sup condition: There is $c > 0$ depending only on 
    $\alpha \in [0,1)$ such that 
    $$
    \inf_{u\neq 0} \sup_{v \neq 0}{ \frac{| B(u,v)|} {\norm{u}_{1,\alpha}\norm{v}_{1,-\alpha}}}
    \geq c >0 .
    $$
    This can be seen by taking, for given $u$ in the infimum, 
    the function $v:= \overline{\widehat{s}}e^{2 r} u$ in the supremum and performing elementary calculations.
    Next, we show that $|(h,\gamma v)_{L^2({\mathbb R}^+\times \{0\})}| \leq C \|v\|_{1,-\alpha}$. 
    This follows also verbatim \cite[Prop.~{6.4.6}]{melenk_book} using~\cite[Lemma~{A.1.8}]{melenk_book}.
    Specifically, by \cite[Lemma~{A.1.8}]{melenk_book} it suffices to 
    ascertain that for $\alpha < \mu \Re (\widehat{s})$ we have 
    \begin{align*}
      \int_{-\omega}^{\omega}{\int_{0}^{\infty}{ r\abs{e^{\alpha r} e^{- \mu \widehat{s}r }}^2 \, +
      r \abs{r \nabla (e^{\alpha r} e^{- \mu \widehat{s}r }) }^2\, dr}d\varphi}
      &\lesssim \int_{0}^{\infty}{ (r+r^3)\abs{e^{2(\alpha - \mu \widehat{s}) r} }\, dr} \\ 
      & < \infty     . 
    \end{align*}
We conclude that the solution $\uhat$ satisfies 
$\|\uhat\|_{1,\alpha} \leq C$ for some constant $C > 0$ depending only
on the choice of $\alpha < \mu \Re(\widehat{s})$. 

\emph{Step 2:}
    For a ball $B_\rho(x)$ of radius $\rho=\bigO(1)$ around any point $x \in \mathcal{C}$ with
    $\operatorname{dist}(x,0)> 2\rho$ we can apply 
    standard elliptic regularity (interior regularity, regularity for homogeneous Dirichlet conditions, and regularity for transmission 
problems---see, e.g., \cite[Lemmas~{5.5.5}, {5.5.7}, {5.5.8}]{melenk_book}) to get 
    \begin{equation}
\label{eq:foobar}
    \norm{\uhat}_{H^2(B_{\rho}(x) \cap \mathcal{C}^{\prime})}\lesssim \norm{ \uhat }_{H^1(B_{2\rho}(x) \cap \mathcal{C})} \lesssim
    e^{-\alpha \operatorname{dist}(x,0)-2\rho} \norm{\uhat}_{1,\alpha}.
    \end{equation}
    A Besicovich covering argument (see, e.g., \cite[Lemma~{4.2.14}]{melenk_book} for details) by such balls and local trace estimates show
    that for $\mathcal{C}_{\infty}:={\mathcal{C}} \setminus \overline{B_1(0)}$
    \begin{align}
      \label{eq:estimates_on_cone_asymptotic}
    \norm{\partial_n \uhat}_{L^2(\partial \mathcal{C}_{\infty})}    
    + \norm{\partial_n \uhat}_{L^1(\partial \mathcal{C}_{\infty})}
      \lesssim \|\uhat\|_{1,\alpha} \lesssim 1, 
    \end{align}
    where the implied constant depends only on $\alpha$, $\omega$, and $\widehat{s}$, $\mu$. 

\emph{Step 3:}
    We show that
    $\norm{\partial_n \uhat}_{L^2(\partial \mathcal{C} \cap B_1(0))}< \infty$. 
    Fix a cut-off function $\chi$ with $\chi \equiv 1$ on $B_1(0)$  and 
    $\operatorname{supp}(\chi) \subseteq B_2(0)$.
We consider the following lifting of the jump $h$ using a single layer 
potential for the Laplacian: 
\begin{equation*}
j(x_1,x_2):= - \frac{1}{2\pi} \int_{\xi=0}^2 
\ln |(x_1,x_2) - (\xi,0)| h(\xi,0)\,d\xi.  
\end{equation*}
Since $h \in L^2({\mathbb R}^+ \times \{0\} \cap B_2(0))$, we have by 
the mapping properties of the single layer potential 
    (see~\cite[Thm.~{2.17}]{acta_numerica12}) that 
$j|_{\partial( \mathcal{C} \cap B_2(0))} 
\in H^1(\partial\mathcal{C} \cap B_2(0))$ with 
\begin{equation*}
\|j\|_{H^1(\partial\mathcal{C} \cap B_2(0))} \leq C
\end{equation*}
for some $C > 0$ depending on $\mu$ and $\hat s$. The jump relations
of the single layer operator provide 
(see, e.g., \cite[Thm.~{6.11}]{book_mclean})
    $\llbracket{\gamma j \rrbracket}=0$ and
    $\llbracket{\partial_n  j}\rrbracket=h$ 
on $(0,2) \times \{0\}$. Since 
$-\Delta j = 0$ and $j \in H^1_{loc}({\mathbb R}^2)$, 
and $\operatorname{supp} \chi \subset B_2(0)$ 
we see that $\widetilde u:= \chi (\uhat - j)$ is 
the $H^1$-function solving 
\begin{align*}
-\Delta \widetilde u + \hat s^2 \widetilde u &=  - \hat s^2 \chi \,j
+ 2 \nabla \chi \cdot \nabla (\uhat -j) + \Delta \chi (\uhat -j) =:\widetilde f
\qquad 
\mbox{ in $\mathcal{C} \cap B_2(0)$},  \\
  \widetilde u &= - j\chi\quad \mbox {on $\partial\mathcal{C} \cap B_2(0)$},
                 \quad \text{and} \quad
 \widetilde u = 0 \quad \mbox{ on $\mathcal{C} \cap \partial B_2(0)$.}
\end{align*}
Since the right-hand side $\widetilde f \in L^2(\mathcal{C} \cap B_2(0))$
and the Dirichlet boundary conditions are in 
$H^1(\partial(B_2(0) \cap \mathcal{C}))$, standard elliptic regularity
theory 
(see, e.g., \cite[Thm.~{4.24}]{book_mclean}) shows $\partial_n \widetilde u \in L^2(\partial(B_2(0)\cap\mathcal{C}))$. 

\emph{Step 4 (proof of (\ref{item:lemma:transmission_on_cone_dimensionless-i})):} 
The 2D Sobolev embedding theorem $H^2 \subset L^\infty$ 
and \eqref{eq:foobar} 
show the desired estimate for $\ell = 0$. 
The argument leading to \eqref{eq:foobar} can be iterated and thus 
yields the stated estimates for any fixed $\ell$. 

\emph{Step 5:}
    Inspection of the proof reveals that all constants (if at all) depend continuously on $\shat$. Since we are only interested in $\shat$ in a compact set 
  determined by the constants from the definition of
    $\mathscr{S}$ we can make all the constants independent of $\hat s$. 
  \end{proof}

  Having studied the transmission problem in a dimensionless form 
in Lemma~\ref{lemma:transmission_on_cone_dimensionless}, we can 
  transfer the results to the setting we actually require 
  using a scaling argument.
  \begin{lemma}
    \label{lemma:estimates_on_cone1}
    Fix $\omega \in (0,\pi)$.
    For $s \in \mathscr{S}$ and $\mu>0$, let $u \in H^1(S_\omega)$ solve the transmission problem on the
    sector $S_{\omega}$:
    \begin{equation}
\begin{split}
      \label{eq:transmission_on_cone}
      -\laplace u + s^2 u & =0 \text{ in } S_\omega \setminus (0,1)\times \{0\}, \\
      \ \normaljump{ u}&=s\,e^{-s \mu\,x_1} \;\text{on $(0,1) \times\{0\}$}, \qquad 
      u=0 \;\text{on $\partial S_{\omega}$}.
\end{split}
    \end{equation}
    Recall that $\Gamma_S$ denotes the parts of $\partial S_{\omega}$ adjacent to the origin.
    Then 
      \begin{align}
        \norm{\partial_n u}_{L^2(\Gamma_S)} \leq C \abs{s}^{1/2} \quad \text{ and } \quad
        \label{eq:asymptotics_on_cone1} 
        \norm{\partial_n u}_{L^1(\Gamma_S)} \leq C . 
      \end{align}
      The constants $C>0$ depend only on $\omega$, 
      the parameter $\mu$, 
      and the choices of $\sigma_0$ and $\delta$ in the definition of $\mathscr{S}$.
  \end{lemma}
  \begin{proof}

We denote by $\Gamma_{1} = \partial B_1(0) \cap \partial S_\omega$ the 
circular arc that is part of $\partial S_\omega$. 
    Write $s=  |s| \hat s$ with $\hat s \in\{\hat s \in \mathscr{S}\,|\, |\hat s|= 1\}$. 
    Let $\uhat$ be the function solving 
    \begin{align*}
      -\laplace \uhat + \shat^2\uhat =0 \text{ in } \mathcal{C}',
      \quad \normaljump{\hat{u}}=e^{- \shat \mu r} \; \text{on $\R_+ \times\{0\}$}, \quad
      \uhat=0 \;\text{on $\partial \mathcal{C}$}
    \end{align*}
    that is given by Lemma~\ref{lemma:transmission_on_cone_dimensionless}. 
    Then we define $u_1(x):=\uhat(|s| \,x)$.   
    Lemma~\ref{lemma:transmission_on_cone_dimensionless} 
     and a simple scaling argument
    gives the following estimates for $u_1$ (for any $j$): 
      \begin{align*}
        \norm{\partial_n u_1}_{L^2(\Gamma_S)} \lesssim \abs{s}^{1/2}, \quad  
        \norm{\partial_n u_1}_{L^1(\Gamma_S)} \lesssim 1, \quad
        \|u_1\|_{W^{j,\infty}(\Gamma_1\setminus (1,0))} \lesssim |s|^j e^{-\alpha_j |s|}.   
      \end{align*}
      The remainder $\delta:= u-u_1 \in H^1(S_\omega)$ then solves 
\begin{align*}
& -\Delta \delta + s^2 \delta = 0 \quad \mbox{ in $S_\omega$}, 
\qquad \delta|_{\Gamma_S} = 0, \qquad \delta|_{\Gamma_1} =  - u_1|_{\Gamma_1}. 
\end{align*}
We note that for this Dirichlet problem with piecewise smooth data that 
are exponentially small in $|s|$,  Lemma~\ref{lemma:lifting_and_apriori}
gives $\norm{\delta}_{|s|,1,S_\omega} \lesssim \|u_1\|_{|s|,1/2,\Gamma_1}$, 
which is exponentially small in $|s|$.
Applying \cite[Thm~{4.24}]{book_mclean} then gives,
  since $u_1$ vanishes on $\partial \Gamma$:
  \begin{align*}
    \|\partial_n \delta\|_{L^1(\Gamma_S)}& \lesssim \|\partial_n \delta\|_{L^2(\Gamma_S)}
    \lesssim |s| \norm{\delta}_{|s|,1,S_{\omega}} + \norm{u_1}_{H^1(\Gamma_1)}
  \end{align*}

which is again exponentially small in $|s|$. The estimate 
        \eqref{eq:asymptotics_on_cone1} follows. 
  \end{proof}

We need the following modification of \cite[Lemma~{3.13}]{MPW17}.
\begin{proposition}
  \label{prop:interpolation_with_log}
  Let $\mathcal{O}$ be a Lipschitz domain.
  Define the Besov space (cf.\ \eqref{eq:besov-norms})
  $$
  B^{1/2}_{2,1}(\mathcal{O}):=\big[L^2(\mathcal{O}),H^1(\mathcal{O})\big]_{1/2,1}.
  $$   
  For $\varepsilon>0$ and every $w \in H^{1/2}(\mathcal{O})$, there exists a function
  $w_{\varepsilon} \in H^1(\mathcal{O})$ with 
  \begin{align*}
  \varepsilon^{- 1/2} \norm{w - w_{\varepsilon}}_{L^2(\mathcal{O})} 
    &\leq C \norm{w}_{H^{1/2}(\mathcal{O})}
      \quad \text{and}  \quad
  \norm{w_{\varepsilon}}_{B^{1/2}_{2,1}(\mathcal{O})}
   \leq C \big(1+\sqrt{\abs{\log(\varepsilon)}}\big)\norm{w}_{H^{1/2}(\mathcal{O})}.
 \end{align*} 
  The constant  depends only on the domain $\mathcal{O}$.
\end{proposition}
\begin{proof}
  This is essentially \cite[Lemma~{3.13}]{MPW17}. The only modification needed is that
  we consider the $H^{1/2}(\mathcal{O})$-norm on the right-hand side instead of the
  $B^{1/2}_{2,\infty}$-norm, which is the reason for getting penalized by a factor $\sqrt{\abs{\log(\varepsilon)}}$ instead of a factor 
   $\abs{\log(\varepsilon)}$ as in \cite[Lemma~{3.13}]{MPW17}.
  The result follows from the same proof, only noting that one can bound using the Cauchy-Schwarz inequality
  \begin{multline*}
  \int_{\varepsilon}^{1}{t^{-\theta}
    \inf_{v \in H^1}\big(\norm{w-v}_{L^2(\mathcal{O})} + t {\norm{v}_{H^1(\mathcal{O})}}\big)
    \frac{dt}{t}} \\
  \leq
    \Big(\int_{\varepsilon}^{1}{t^{-2\theta}
      \inf_{v \in H^1}\big(\norm{w-v}_{L^2(\mathcal{O})} + t {\norm{v}_{H^1(\mathcal{O})}}\big)^2
      \frac{dt}{t}}\Big)^{1/2}
  \Big(\int_{\varepsilon}^{1}{\frac{dt}{t}}\Big)^{1/2},
\end{multline*}
and the last factor produces a factor $1+\sqrt{|\log(\varepsilon)|}$.
\end{proof}

\begin{lemma}
  \label{lemma:apriori_on_cone}
  Let $u$ solve \eqref{eq:transmission_on_cone}.
  Then,  there exists a constant $C>0$ 
  depending only on $\omega$, $\mu$ and the parameters in the definition of $\mathscr{S}$ such that 
  \begin{align*}
    \norm{\partial_n u}_{H^{-1/2}(\Gamma_S)}
    &\leq C \sqrt{\log(\abs{s}+2)}. 
  \end{align*}
\end{lemma}
\begin{proof}
  For $w \in H^{1/2}(\Gamma)$,
  select $w_{\varepsilon}$ as in Proposition~\ref{prop:interpolation_with_log} with $\varepsilon>0$ to
  be chosen later. We calculate for
  $\Gamma_{\pm \omega}$, i.e., the two parts of $\partial {\Omega}$ adjacent to the origin: 
  \begin{align*}
    \big|{\big({\partial_n u,w}\big)_{L^2(\Gamma_{\pm\omega})}}\big|
    &\leq \big|{\big({\partial_n u,w -w_{\varepsilon}}\big)_{L^2(\Gamma_{\pm\omega})}}\big|
      +\big|{\big({\partial_n u,w_{\varepsilon}}\big)_{L^2(\Gamma_{\pm\omega})}}\big| \\
    &\leq \norm{\partial_n u}_{L^2(\Gamma_{\pm\omega})}\norm{w - w_{\varepsilon}}_{L^2(\Gamma_{\pm\omega})}
      +  \norm{\partial_n u}_{L^1(\Gamma_{\pm \omega})}\norm{w_{\varepsilon}}_{L^\infty(\Gamma_{\pm \omega})}.
  \end{align*}
  Since $\Gamma_{\pm \omega}$ is a one-dimensional line segment, 
we can use the Sobolev embedding~\cite[Sec.~{32}]{tartar}
  to estimate 
  $$
  \norm{w_{\varepsilon}}_{L^{\infty}(\Gamma_{\pm \omega})} \lesssim \norm{w_{\varepsilon}}_{B^{1/2}_{2,1}(\Gamma_{\pm \omega})}.
  $$
  Overall, we get using the properties of $w_{\varepsilon}$ from Proposition~\ref{prop:interpolation_with_log}
  and the estimates on $\partial_n u$ from~\eqref{eq:estimates_on_cone_asymptotic}:
  \begin{align*}
     \big|{\big({\partial_n u,w}\big)_{L^2(\Gamma_{\pm \omega})}}\big| 
    & \lesssim 
      \norm{\partial_n u}_{L^2(\Gamma_{\pm \omega})}\varepsilon^{1/2} \norm{w}_{H^{1/2}(\Gamma_{\pm \omega})}+
      (1\!+\!\!\sqrt{\abs{\log(\varepsilon)}})\norm{\partial_n u}_{L^1(\Gamma_{\pm \omega})}\norm{w}_{H^{1/2}(\Gamma_{\pm \omega})}\\
    &\lesssim
      \Big(1 + \abs{s}^{1/2} \varepsilon^{1/2}  +\sqrt{\abs{\log(\varepsilon)}}\Big)\norm{w}_{H^{1/2}(\Gamma_{\pm \omega})}.
  \end{align*}
  Choosing $\varepsilon = \abs{s}^{-1}$ completes the proof.
\end{proof}

We are now in a position to study a more general transmission problem, 
namely, allowing for Dirichlet jumps and more general Neumann transmission data.
\begin{lemma}[Transmission problem]
  \label{lemma:transmission_lifting}
  Let $\mathcal{O}\subset \R^2$ be an open Lipschitz domain.
  Let $\Gamma' \subset \mathcal{O}$ be a smooth interface
  that splits $\mathcal{O}$ into two disjoint Lipschitz domains $\mathcal{O}_{1}$
  and $\mathcal{O}_{2}$.

  Given $g \in \widetilde{H}^{1/2}(\Gamma')$, $h \in H^{-1/2}(\Gamma')$,  
  there exists a unique solution $u \in H^1(\mathcal{O}_1 \cup \mathcal{O}_2)$ to the following problem:
  \begin{align*}
    -\laplace u + s^2 u &= 0 \quad\text{in $\mathcal{O}_1 \cup \mathcal{O}_2$},\\ 
\gamma^- u& =0 \quad \text{on $\partial \mathcal{O}$,} \qquad
    \tracejump{u}=g \, \text{ and }\,  \normaljump{u}=h \quad \text{across } \Gamma'.
  \end{align*}
  Additionally, the following estimate holds:
  \begin{align}
\label{eq:lemma:transmission_lifting-10}
    \norm{u}_{\abs{s},1,\mathcal{O}}&\lesssim\norm{g}_{\abs{s},1/2,\Gamma'} + \|d_{\partial \Gamma'}^{-1/2} g\|_{L^2(\Gamma')}+ \norm{h}_{\abs{s},-1/2,\Gamma'}.
  \end{align}  
  If $\mathcal{O}_1$ and $\mathcal{O}_2$ are polygons,
     $\Gamma'$ is a straight line, and $h$ can be decomposed as
  $$
  h(x)=h_1 s e^{-\mu s \abs{x}} + h_2(x)
  $$
  for some $h_1\in \R$, $\mu >0$, and $h_2 \in H^{-1/2}(\Gamma')$, then
$\partial_n u$ exists pointwise almost everywhere and 
\begin{align}
  \nonumber
 \|\partial_n u\|_{H^{-1/2}(\partial\mathcal{O})} 
&\lesssim 
 |s|^{1/2} \left(
\norm{g}_{\abs{s},1/2,\Gamma'} + \|d_{\partial \Gamma'}^{-1/2} g\|_{L^2(\Gamma')}+
\norm{h_2}_{\abs{s},-1/2,\Gamma'} \right)  
\label{eq:lemma:transmission_lifting-20} \\
 &\qquad \mbox{}
+ \abs{h_1} \sqrt{\log(\abs{s}+2)}.
\end{align}

\end{lemma}
\begin{proof}
\emph{Proof of \eqref{eq:lemma:transmission_lifting-10}:}
  Since $g$ is assumed in $\widetilde{H}^{1/2}(\Gamma')$,
  we can extend it by $0$ to a function $\widetilde{g} \in H^{1/2}(\partial \mathcal{O}_1)$
  such that (see for example~\cite[Thm.~{3.33}]{book_mclean})
  $$\norm{\widetilde{g}}_{\abs{s},1/2,\partial \mathcal{O}_1}
  \lesssim \norm{g}_{\abs{s},1/2,\Gamma'} + \|d_{\partial \Gamma'}^{-1/2} g\|_{L^2(\Gamma')}. 
  $$

  We solve a Dirichlet problem on $\mathcal{O}_1$ with data $\widetilde{g}$ to obtain $u_1$ and extend it by $0$ to
  $\mathcal{O}_2$. Then we solve the following problem
  on $\mathcal{O}$: Find $u_2 \in H_0^1(\mathcal{O})$ such that 
  \begin{align*}
    & (\nabla u_2,\nabla v)_{L^2(\mathcal{O})} + s^2 (u_2,v)_{L^2(\mathcal{O})}=  
     \langle{h},{\gamma_{{\Gamma'}} v}\rangle_{\Gamma'} - (\nabla u_1,\nabla v)_{L^2(\mathcal{O}_1)} - s^2 (u_1,v)_{L^2(\mathcal{O}_1)} 
\quad \forall v \in H_0^1(\mathcal{O}),
  \end{align*}
  where $\gamma_{\Gamma'}$ denotes the trace operator on $\Gamma'$.
  The function $u:=u_1+u_2$ then solves the transmission problem. 
The estimate 
\eqref{eq:lemma:transmission_lifting-10}
follows from Lemmas~\ref{lemma:lifting_and_apriori} to bound $u_1$
and, in order to bound $u_2$,  the ellipticity of the sesquilinear form (see Lemma~\ref{lemma:well_posedness})
together with the trace estimate (Lemma~\ref{lemma:bounded-inverse})
to estimate the contribution $\langle{h},{\gamma_{{\Gamma'}} v}\rangle_{\Gamma'}$.

\emph{Proof of \eqref{eq:lemma:transmission_lifting-20} for $h_1 = 0$:}
Introduce 
$\mathcal{V}:= 
\partial\Gamma^\prime \cup \{\text{vertices of $\Omega$}\}$.  
Since $\mathcal{O}_1$ and $\mathcal{O}_2$ are piecewise smooth 
and $u = 0$ on $\partial\mathcal{O}$ and the right-hand side is homogeneous,
the solution $u$ is smooth up to the boundary with the exception of the 
vertices of $\Omega$ and near the interface $\Gamma^\prime$. 
Hence, $\partial_n u$ exists pointwise everywhere on 
$\partial\mathcal{O}\setminus \mathcal{V}$.  
To show the estimate \eqref{eq:lemma:transmission_lifting-20}, 
we consider test functions 
$v \in V:= \{v \in C^\infty(\overline{\mathcal{O}})\,|\, v \text{ vanishes in a
neighborhood of $\mathcal{V}$}\}$. Since $\partial_n u$ exists pointwise 
on $\partial\mathcal{O}\setminus\mathcal{V}$ and $v \in V$ vanishes in a neighborhood
of $\mathcal{V}$, the duality pairing 
$\langle \partial_n u,v\rangle_{\partial\mathcal{O}}$ is well-defined and 
an integration by parts gives 
\begin{equation}
\label{eq:lemma:transmission_lifting-30}
\langle \partial_n u,v\rangle_{\partial\mathcal{O}}
 = (\nabla u,\nabla v)_{\mathcal{O}_1 \cup \mathcal{O}_2} + 
s^2 (u,v)_{\mathcal{O}} - 
\langle h,\gamma_{\Gamma'} v\rangle_{\Gamma^\prime}. 
\end{equation}
Since $V$ is dense in $H^1(\Omega)$ (because $\mathcal{V}$ consists
of finitely many points), the equation 
(\ref{eq:lemma:transmission_lifting-30}) actually holds for all
$v \in H^1(\Omega)$. Given $\xi \in H^{1/2}(\partial\mathcal{O})$ we 
select $v_\xi \in H^1(\Omega)$ with $v|_{\partial\mathcal{O}} = \xi$ as the 
lifting given by (\ref{eq:energy-minizing-extension})
in Lemma~\ref{lemma:trace_estimates}, which satisfies 
$\|v_\xi\|_{|s|,1,\mathcal{O}} \lesssim |s|^{1/2} \|\xi\|_{H^{1/2}(\partial\mathcal{O})}$. This implies 
$$
|\langle\partial_n u, \xi\rangle_{\mathcal{O}} | 
\leq \|u\|_{|s|,1,\mathcal{O}_1 \cup \mathcal{O}_2} \|v_\xi\|_{|s|,1,\mathcal{O}}
+ \|h\|_{|s|,-1/2,\Gamma'} \|v_\xi\|_{|s|,1/2,\Gamma'}
$$
By the trace theorem we have $\|v_\xi\|_{|s|,1/2,\Gamma'} \lesssim
\|v_\xi\|_{|s|,1,\mathcal{O}}$. 
Taking the supremum over all $\xi \in H^{1/2}(\partial\mathcal{O})$ yields
(\ref{eq:lemma:transmission_lifting-20}) for the case $h_1 = 0$. 

\emph{Proof of \eqref{eq:lemma:transmission_lifting-20} for $h_1 \ne 0$:}
If $h_1\neq 0$, we lift this contribution separately by solving
the corresponding transmission problem \eqref{eq:transmission_on_cone}.
The estimate follows by applying Lemma~\ref{lemma:apriori_on_cone}
and a suitable localization.
%
\end{proof}

\subsubsection{Corner layers}

Before we can bound the functions used to match boundary layers, we must control the jump between two boundary layer solutions.
We start with a very simple geometric situation.
\begin{lemma}
  \label{lemma:simple_corner_jumps}
  Fix $\omega \in (0,\pi)$.
  Consider the sector $S_{\omega}$ 
  and let $g \in H^1(-1,1)$.  
  Consider the two Cartesian coordinate systems
    $(\zeta_1,\zeta_2)$ and $(\widehat{\zeta}_1,\widehat{\zeta}_2)$,  
    each given by one of the straight sides of the sector and such that the components $\zeta_2$ and $\widehat\zeta_2$ of
the bisector $\{(r,0)\colon r \in (0,1)\}$ are positive.
    In polar coordinates $(r,\varphi)$ (with $r > 0$, $\varphi \in (-\omega,\omega)$) these coordinates are given by
  \begin{align}
    \label{eq:polar-coords}
    \colvec{\zeta_1\\ \zeta_2}
    &=\colvec{r\cos(\varphi+\omega)\\ r \sin(\varphi+\omega)},
    \qquad \colvec{\widehat{\zeta}_1\\\widehat{\zeta}_2}
    =\colvec{-r\cos(\varphi-\omega) \\ -r \sin(\varphi-\omega )}. 
  \end{align}
    
  For $\mu >0$ define
  \begin{align*}
    u_1(\zeta_1,\zeta_2):=g(\zeta_1)e^{- \mu s \zeta_2} \qquad \text{ and } \qquad u_2(\widehat{\zeta}_1,\widehat{\zeta}_2):=g(\widehat{\zeta}_1)e^{- \mu s \widehat{\zeta}_2}.
  \end{align*}
  
\begin{enumerate}[(i)]
\item 
  \label{item:lemma:simple_corner_jumps-i}
  On the line segment $\Gamma':=\{(r,0)\colon r \in (0,1)\}$ 
  the following estimates hold with $d_0:=\operatorname{dist}(\cdot,(0,0)) = r$
  \begin{align}
    \!\!\! 
\norm{u_1 - u_2}_{\abs{s},1/2,\Gamma'} \!+\! \norm{d_{0}^{-1/2}(u_1 - u_2)}_{L^2(\Gamma')}
    &\! \lesssim\! \abs{s}^{-1/2}\norm{g}_{H^{1}(-1,1)}.
      \label{eq:corner_dirichlet_jump}
  \end{align}
\item 
  \label{item:lemma:simple_corner_jumps-ii}
  The normal jump across $\Gamma'$ can be decomposed as 
  \begin{align*}
    \partial_n^- u_1(x) - \partial_n^- u_2(x)& =
    h_1 sg(0) e^{-\mu s \abs{x}} + h_2(x)  \;\text{ with} \;
    \norm{h_2}_{\abs{s},-1/2,\Gamma'}
    \lesssim
      \abs{s}^{-1/2}\norm{g}_{H^{1}(-1,1)} ,
  \end{align*}
  where the orientation of the normal is arbitrarily fixed and $h_1 \in \R$ is
  independent of $g$. 
\item 
  \label{item:lemma:simple_corner_jumps-iii}
There holds $\|e^{-\mu s |x|}\|_{H^{-1/2}(\Gamma')} \lesssim |s|^{-1}$. 
\end{enumerate}
\end{lemma}
\begin{proof}
  We work in polar coordinates, which are related to the coordinates 
$(\zeta_1,\zeta_2)$ and $(\widehat\zeta_1,\widehat\zeta_2)$ by  (\ref{eq:polar-coords}).
  For brevity of notation we introduce the constants $c_1:=\cos(\omega)$, $c_2:=\sin(\omega)$ 
  and note $c_2>0$.  

\emph{Proof of (\ref{item:lemma:simple_corner_jumps-i}):}
  We start with the estimate for the Dirichlet jump and calculate on $\Gamma'$:
  \begin{align*}
    \tracejump{u}(r,\varphi)
    &:=u_1(\zeta_1,\zeta_2)-u_2(\widehat{\zeta}_1,\widehat{\zeta}_2) 
    =\left[g(r c_1) - g(-r c_1)\right]e^{-c_2 \mu s r}.
  \end{align*}
  We estimate:
  \begin{align}
\nonumber 
     \norm{\tracejump{u}}_{L^2(\Gamma')}^2
    &=\int_0^1{\left[g(r c_1)) - g(-r c_1)\right]^2 \,e^{-2 \Re(s) \mu c_2 r} \,dr} \\
    & =\int_{0}^{1}{\left[\int_{-r}^{r}{g'(\tau c_1)} c_1 \,d\tau\right]^2 \,e^{-2 \Re(s) \mu c_2 r} \,dr}  
      \label{eq:tmp1}
\\
\nonumber 
    & \lesssim \int_{0}^{1}{\norm{g'}^2_{L^2(-1,1)} r \,e^{-2 \Re(s) \mu c_2 r}\,dr}
      \lesssim   \frac{1}{\Re(s)^2}  \norm{g'}_{L^2(-1,1)}^2.
  \end{align}
  An analogous computation gives:
  \begin{align*}
    \norm{d_{0}^{-1/2}\tracejump{u}}_{L^2(\Gamma')}^2
    \lesssim  \frac{1}{\Re(s)}  \norm{g'}_{L^2(-1,1)}^2.
  \end{align*}
  Next we compute the tangential derivative of $\tracejump{u}$ on $\Gamma'$:
  \begin{align*}
    \frac{\partial}{\partial r}  \tracejump{u}&= 
    -s \mu c_2 e^{-c_2 \mu s r } \big[g(r c_1) - g(-r c_1)\big] 
                                                + e^{-c_2 \mu s r } c_1\big[g'(r c_1) + g'(- r c_1)\big].                                                
  \end{align*}
  The first term is handled analogously to the $L^2$-term. For the second term we use the crude estimate $\abs{e^{-s  \mu r c_2}}\lesssim 1$
  and get:
  \begin{align}
    \label{eq:tmp2}
    \big\|\frac{\partial}{\partial r} \tracejump{u}\big\|_{L^2(\Gamma')} \lesssim \norm{g'}_{L^2(-1,1)}.
  \end{align}
  Interpolating \eqref{eq:tmp1} and~\eqref{eq:tmp2} then gives~\eqref{eq:corner_dirichlet_jump}.
  
\emph{Proof of (\ref{item:lemma:simple_corner_jumps-ii}):}
  In polar coordinates, the normal derivative on $\Gamma^\prime$ of a function 
is (up to the sign) given by
$ \partial_{n_{\Gamma'}} u = r^{-1}\frac{\partial u}{\partial \varphi}$
  . Thus it is sufficient to estimate the angular derivatives.

  On $\Gamma'$, we calculate for the angular derivative:
  \begin{align*}
    \frac{1}{r}\frac{\partial}{\partial \varphi}{\big(u_1 - u_2\big)}
    &=-c_2 \left[g'(r c_1) - g'(-r c_1)\right] c_2 e^{-s \mu c_2 r} 
      - s \left[g(r c_1) + g(-r c_1)\right] \mu c_1  e^{-s \mu c_2 r}.
  \end{align*}
  
  After substracting the contribution $- 2 s g(0) \mu c_1 e^{-s \mu c_2 r}=:
  h_1 s g(0) e^{-s \mu c_2 r}$ from the second term,  
  the two terms are structurally similar to the derivative of $\tracejump{u}$.
Hence, we analogously get for $h_2:=\normaljump{u} - h_1 s g(0) e^{-s \mu c_2 r}$:
  \begin{align*}
    \norm{h_2}_{L^2(\Gamma')} \lesssim \norm{g}_{H^1(-1,1)}.
  \end{align*}
  To control $\|h_2\|_{|s|,-1/2,\Gamma}$  we calculate for $\xi \in H^{1/2}(\Gamma')$:
  \begin{align*}
    \left| \dualproduct{h_2}{\xi}\right|
    &\lesssim \norm{h_2}_{L^2(\Gamma')}\norm{\xi}_{L^2(\Gamma')}
      \lesssim \abs{s}^{-1/2} \norm{h_2}_{L^2(\Gamma')} \norm{\xi}_{\abs{s},1/2,\Gamma'}.
  \end{align*}
\emph{Proof of (\ref{item:lemma:simple_corner_jumps-iii}):} 
We identify $\Gamma'$ with the interval $(0,1)$. 
A direct calculation shows $\|e^{-\mu s r}\|_{L^2(0,1)} \lesssim |s|^{-1/2}$. 
A test function $v \in H^1_0(0,1)$ can be represented as 
$v(x) = \int_0^r v^\prime(t)\,dt$. Hence, an integration by parts yields 
$$
\left| \int_0^1 e^{-\mu s r} v(r)\,dr \right| 
= \left| \frac{1}{\mu s} \int_0^1 e^{-\mu s r} v^\prime(r)\,dr \right|
\lesssim  |s|^{-3/2} \|v\|_{H^1(0,1)}. 
$$
Thus, $\|e^{-\mu s r}\|_{H^{-1}(0,1)} \lesssim |s|^{-3/2}$. Furthermore, 
we have $\|e^{-\mu s r}\|_{L^2(0,1)} \lesssim |s|^{-1/2}$. 
Interpolation then
yields $\|e^{-\mu s r}\|_{H^{-1/2}(0,1)} \lesssim |s|^{-1}$. 
\end{proof}

\subsubsection{Decomposing the DtN-operator}
In Section~\ref{sec:smooth-gemoetries} we discussed the DtN-operator
for smooth geometries. Here, we study the case of polygonal domains. 
We will do so
by introducing corner layers, similarly to what was done in~\cite[Sec.~{7.4.3}]{melenk_book}.
\begin{figure}
  \begin{subfigure}[t]{0.49\textwidth}
    \includeTikzOrEps{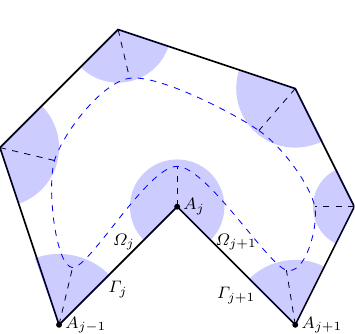} 
    \caption{
      Subdomains for defining the corner layers $u_{CL}$ and boundary layers $u_{BL}$: 
      Solid blue regions indicate $\operatorname{supp}(u_{CL})$. The regions
      $\Omega_j$ on which $u_{BL}$ is defined are confined by the dashed lines.
    \label{fig:geometry_corner_layers} }
  \end{subfigure}
  \hfill
  \begin{subfigure}[t]{0.49\textwidth}
  \centering
  \includeTikzOrEps{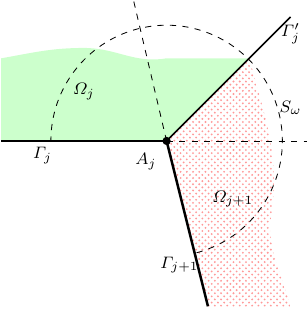}
  \caption{Situation at a corner $A_j$. 
     Marked in color is  the
      support of the cut-off functions $\chi_{BL}$;
      solid green: $\Omega_j$, dotted red: $\Omega_{j+1}$. 
  \label{fig:corner_geometry}}
\end{subfigure}
\caption{Boundary layer and corner layer construction for nonsmooth domains}
\end{figure}

The following Theorem~\ref{thm:decomposition_for_polygons} presents a
decomposition of the DtN-Operator into several contributions. To describe
them, we need some notation as illustrated in Fig.~\ref{fig:corner_geometry}. 
The polygon $\Omega$ has vertices 
$A_1,\ldots,A_J$ and edge $\Gamma_j$ connects $A_j$ with $A_{j+1}$ 
(we set $A_{J+1}:= A_1$ and $\Gamma_{J+1}:= \Gamma_1$ and, for simplicity of notation, we assume that 
$\partial\Omega$ consists of a single component of connectedness). 
$\Gamma_j^\prime$ is the bisector of the angle at vertex $A_j$. 
The subdomains $\Omega_{j}$ are confined by four curves: $\Gamma_{j}$, 
the bisectors at $A_j$ and $A_{j+1}$ (dashed black in Figure~\ref{fig:geometry_corner_layers}), and a fourth curve completely contained
in $\Omega$ and sufficiently close to $\Gamma_{j}$ (dashed blue in  Figure~\ref{fig:geometry_corner_layers}) . 
We set $\Omega_0:= \Omega \setminus \cup_{i=1}^J \overline\Omega_i$ and, for convenience
$\Omega_{J+1}:= \Omega_1$.  We fix
$\chi_{BL} \in C^\infty(\R^2)$ with $\operatorname{supp} \chi_{BL} \subset 
\overline{\cup_i \Omega_i}$ and $\chi_{BL} \equiv 1$ near $\Gamma$. 
Finally, for each vertex $j$ we let $\chi_{CL,j} \in C^\infty(\R^2)$ 
be a cut-off function with $\operatorname{supp} \chi_{CL,j} \cap \{A_{j'} \} = 
\emptyset$ for $j' \ne j$ such that 
$\chi_{CL,j} \equiv 1$ on 
$\Gamma_j^\prime \cap \{x \in \Omega\,|\, 
\chi_{BL}(x)  = 1\}$. 

\begin{theorem}
  \label{thm:decomposition_for_polygons}
  Let $\Omega^- \subseteq \R^2$ be a polygon, $s \in \mathscr{S}$.
Let $g \in H^1(\Gamma)$ with $g|_{\Gamma_i} \in H^2(\Gamma_i)$. 
  Let $u$ solve
  $$
-\laplace u + s^2 u =0 \quad \text{ in $\Omega^-$}, \qquad \gamma^- u =g \quad \text{on $\Gamma$}. 
 $$ 
  Then $u$ can be decomposed as 
  $u=\chi_{BL} u_{BL} + \sum_{j=1}^J \chi_{CL,j} u_{CL,j} + r$ such that
for a $C > 0$ depending only on $\Omega$: 
  \begin{enumerate}[(i)]
  \item
    \label{it:dti_for_polygons}
    $u_{BL}|_{\Omega_i} \in H^2(\Omega_i)$ and 
$\|u_{BL}\|_{|s|,1,\Omega_i} \leq C  |s|^{1/2} \|g\|_{H^1(\Gamma_i)}$ 
for each $i=0,\ldots,J$. Additionally,  
    $\partial_n^- u_{BL} - s \gamma^- u_{BL} =0$.    
  \item
    \label{it:corner_layer_est_for_polygons}
For each $j \in \{1,\ldots,J\}$ the function $u_{CL,j}$ is in $H^1(\Omega_i)$ for each $i=0,\ldots,J$ and 
$u_{CL,j}|_{\Gamma} = 0$. 
Furthermore, $-\Delta u_{CL,j} + s^2 u_{CL,j} = 0$ on 
$\Omega_j \cup \Omega_{j+1}$ and 
$\partial_n \chi_{CL,j} u_{CL,j}$ exists on each edge $\Gamma_i$, $i =1,\ldots,J$, and 
    \begin{align*}
      \norm{\partial_n^- (\chi_{CL,j} u_{CL,j}) - s \gamma^- \chi_{CL,j}u_{CL,j} }_{H^{-1/2}(\Gamma)} \lesssim C \norm{g}_{H^1(\Gamma)}.
    \end{align*}    
Furthermore, $\sum_{j=1}^J \|u_{CL,j}\|_{|s|,1,\Omega_j \cup \Omega_{j+1}} \leq  C \sqrt{\log(s+2)} \|g\|_{H^1(\Gamma)}$. 
  \item
    \label{it:remainder_est_for_polygons}
    The remainder $r$ satisfies 
    \begin{align*}
      \norm{\partial_n^- r - s \gamma^- r }_{H^{-1/2}(\Gamma)} &\leq C \,\Big( \norm{g}_{H^1(\Gamma)} + |s|^{-1} \sum_{j=1}^J \|g\|_{H^{2}(\Gamma_j)}\Big). 
    \end{align*}
  \end{enumerate}
  The analogous statement holds for the exterior problem upon replacing $s$ by $-s$ in~(\ref{it:dti_for_polygons})-(\ref{it:remainder_est_for_polygons}).
\end{theorem}
\begin{proof}  
\emph{1.~step (construction of $u_{BL}$):} 
For each $\Omega_i$, let $(\theta_i,\rho_i)$ be the boundary fitted
coordinates obtained by an affine parametrization of  the line that contains $\Gamma_i$ 
by $\theta_i$ and denoting by $\rho_i$ the (signed) distance from that line. 
Write $\widehat g(\theta_i)$ for the function $g$ on $\Gamma_i$ in
the coordinates $(\theta_i,\rho_i)$ and extend it $H^1$ and $H^2$-stable
to the line. We define, 
in boundary fitted coordinates $(\theta_i,\rho_i)$, the function 
$u_{BL}(\theta_i,\rho_i):=\widehat g(\theta_i) e^{-s \rho_i}$. 
That is, the function $u_{BL}$ is given by applying the construction 
from Lemma~\ref{lemma:asympt_expansion}. We have by construction 
$\partial_n u_{BL} - s u_{BL} = 0$ on $\Gamma$ and 
$\|u_{BL}\|_{|s|,1,\Omega_i} \lesssim \sqrt{|s|} \|g\|_{H^1(\Gamma_i)}$. 

\emph{2.~step (construction of $u_{CL}$):} 
The function $u_{BL}$ is discontinuous across the bisectors 
$\Gamma_j^\prime$. The corner layers $u_{CL,j}$ corrects this. Focussing 
on the bisector $\Gamma_j^\prime$, let 
$\Gamma_j$ and $\Gamma_{j+1}$ be the edges meeting at $A_j$. 
Fix $\kappa$ such that $\chi_{BL} u_{BL} \equiv 0$ on 
${\Gamma^\prime \cap \R^2\setminus B_\kappa(A_j)}$.  
On the sector $S_{\omega}=B_{\kappa}(A_j) \cap \Omega^- $ define
$u_{CL,j}$ as the solution of the following transmission problem: 
  \begin{align*}
    -\laplace u_{CL,j} + s^2 u_{CL,j}  & = 0 \text{ on $S_{\omega} \setminus (0,\kappa)\times\{0\}$}, \qquad \;u=0\text{ on } \partial S_{\omega}, \\
    \tracejump{u_{CL,j}}&=-\tracejump{(\chi_{BL} u_{BL})}, \quad \text{ and } \\
    \normaljump{u_{CL,j}} & = -\normaljump{ (\chi_{BL} u_{BL})} \; \text{ on $\Gamma'_j\cap B_\kappa(A_j)$}.
  \end{align*}
Up to translation and rotation, we are essentially in the setting of 
  Lemmas~\ref{lemma:transmission_lifting} and \ref{lemma:simple_corner_jumps}. 
That is, on $\Omega_j$ and $\Omega_{j+1}$ the function $u_{BL}$ has the 
form given in Lemma~\ref{lemma:simple_corner_jumps} so that (taking
additionally the effect of $\chi_{BL}$ into account) we arrive at
\begin{align*}
& \|\tracejump{(\chi_{BL} u_{BL})}\|_{|s|,1/2,\Gamma_j^\prime \cap B_\kappa(A_j)} 
+ \|r_j^{-1/2} 
\tracejump{(\chi_{BL} u_{BL})}\|_{L^2(\Gamma_j^\prime \cap B_\kappa(A_j))}  \\
  &\qquad \lesssim |s|^{-1/2} \|g\|_{H^1(\Gamma_j \cup \Gamma_{j+1})}, 
\end{align*}
where $r_j = \operatorname{dist}(\cdot,A_j)$.
For the normal derivative,  
Lemma~\ref{lemma:simple_corner_jumps} provides the representation 
$\normaljump{ (\chi_{BL} u_{BL}) }(x)= h_1  s g(0) e^{-s \mu \abs{x}}+ h_2(x)$ with 
\begin{align*}
  \abs{h_1 g(0)}&\lesssim  \|g\|_{H^1(\Gamma_j \cup \Gamma_{j+1})} \text{ and }  
              \|h_2\|_{|s|,-1/2,\Gamma_j^\prime \cap B_\kappa(A_j)} \lesssim |s|^{-1/2} \|g\|_{H^1(\Gamma_j \cup \Gamma_{j+1})}.  
\end{align*}
By Lemma~\ref{lemma:transmission_lifting}, we therefore get 
\begin{align*}
\|u_{CL,j}\|_{|s|,1,\Omega_j \cap B_{\kappa}(A_j)} + 
\|u_{CL,j}\|_{|s|,1,\Omega_{j+1} \cap B_{\kappa}(A_j)} 
\lesssim  \|g\|_{H^1(\Gamma_j \cup \Gamma_{j+1})}. 
\end{align*}
Furthermore, Lemma~\ref{lemma:transmission_lifting} provides for 
$\partial_n (\chi_{CL,j} u_{CL,j})$ on 
$(\Gamma_j \cup \Gamma_{j+1} ) \cap B_{\kappa}(A_j)$ the bound 
\begin{align*}
\|\partial_n (\chi_{CL,j} u_{CL,j}) \|_{H^{-1/2}(\Gamma_j \cup \Gamma_{j+1})} 
\lesssim \sqrt{\log(\abs{s}+2)}\|g\|_{H^1(\Gamma_j \cup \Gamma_{j+1})}. 
\end{align*} 
 Noting that $u_{CL,j} = 0$ on $\Gamma$, the asserting of 
    (\ref{it:corner_layer_est_for_polygons}) follows.

\emph{3.~Step (Construction of $r$):} 
The function $r \in H^1_0(\Omega^-)$ is defined as
$r:= u - u_{BL} - \sum_{j=1}^J \chi_{CL,j} u_{CL,j}$. It satisfies the 
equation 
  $
  - \laplace r + s^2 r= f
  $
  with $f$ satisfying 
$$
\|f\|_{L^2(\Omega^-)}  \lesssim |s|^{1/2} \|g\|_{H^1(\Gamma)} + 
|s|^{-1/2} \sum_{i=1}^J \|g\|_{H^2(\Gamma_i)}. 
$$
The bounds of Lemma~\ref{lemma:trace_estimates} and \ref{lemma:lifting_and_apriori} then conclude the proof. 

\emph{4.~Step (exterior domains):} 
  The result for the exterior problem follows along the same lines.
\end{proof}
{}From Theorem~\ref{thm:decomposition_for_polygons} we deduce the 
following result for the DtN operator under slightly lower regularity
requirements: 
\begin{corollary}
\label{cor:decomposition_for_polygon}
Let $\Omega\subset \R^2$ b a polygon and $s \in \mathscr{S}$. 
Let $g \in H^1(\Gamma)$ and $u$ solve 
$$
-\Delta u + s^2 u = 0 \quad \mbox{ in $\Omega^-$}, 
\qquad \gamma^- u = g \quad \mbox{ on $\Gamma$.} 
$$
Then 
$$
\|\partial_n u - s u\|_{H^{-1/2}(\Gamma)} \lesssim \sqrt{\log(\abs{s}+2)}\|g\|_{H^1(\Gamma)}. 
$$
The analogous statement holds for the exterior problem upon replacing $s$ by $-s$.
\end{corollary}
\begin{proof}
We employ the smoothing technique as in 
Corollary~\ref{cor:asymptotic_exp_after_smoothing}: By smoothing, 
on a length scale $|s| \ge 1$ or by interpolation, one can construct a function 
$\widetilde g \in H^1(\Gamma)$ such that $\widetilde g \in H^2(\Gamma_i)$
for each edge $\Gamma_i$ and such that 
$$
|s|^{-1/2} \sum_{i=1}^J \|\widetilde g\|_{H^2(\Gamma_i)} 
+ |s| \|g - \widetilde g\|_{L^2(\Gamma)} + \|\widetilde g\|_{H^1(\Gamma)}
\lesssim \|g\|_{H^1(\Gamma)}. 
$$
This can be done in two steps: first, one defines the approximation
edgewise and in a second step ensure continuity at the vertices of $\Omega$
by introducing an appropriate correction, e.g., by a piecewise linear function.
The remainder of the proof is then as in Corollary~\ref{cor:asymptotic_exp_after_smoothing}.
\end{proof}

  

\section{Numerical Examples}
\label{sect:numerics}
In this section, we compare the performance of the numerical schemes of Theorem~\ref{thm:convergence_cq} 
with the more standard method of Proposition~\ref{prop:standard_method} for an interior scattering problem.
That is, we compare the Runge-Kutta convolution quadrature approximation by the following two methods: 
\begin{itemize}
\item[$\bullet$]
 $[\dtn^-(\dd)] {\uinc}$, which is denoted ``standard method'', and   
\item[$\bullet$] 
 $[\ddinv \dtn^-(\dd)] {\duinc}$, which is denoted ``differentiated method''. 
\end{itemize}
%

We use two different Runge-Kutta methods of the Radau IIA family, one with 3 and one with 5 stages.
For the 3-stage version, we have $q=3$ and $p=5$. We therefore expect a convergence rate
of order $3$ for the standard method and full classical order $5$
up to logarithmic terms for the differentiated scheme.

In order to show that our theoretical estimates are sharp, we also look at the 5-stage method. There,
the stage order is $q=5$ and the classical order $p=9$. The expected rates are therefore $5$ and $7$ respectively
for the two numerical schemes up to logarithmic terms.

For simplicity, we consider the interior scattering problem and prescribe an exact solution as
the travelling wave
$$
u(x,t):=\psi(\mathbf{d}\cdot x - t) \qquad \text{with} \quad \psi(\tau):=\cos\left(\frac{\pi\,\tau}{2}\right)\, e^{-\frac{(\tau - \tau_0)^2}{\alpha}}.
$$
The wave direction is selected as $\mathbf{d}:=(\frac{1}{\sqrt{2}},\frac{1}{\sqrt{2}})$, 
and the other parameters were $\tau_0:=4$ and $\alpha:=0.05$. We integrated until the end time $T=12$.
In order to show that the method works with the predicted rates, even for non-convex geometries, we consider the classical L-shaped geometry,
given by the vertices
$$
  (0.5,0),\;
  (1,0),\;
  (1,1),\:
  (0,1),\;
  (0,0.5),\;
  (0.5,0.5).
$$

As the space discretization, we employ a Galerkin boundary element method of order $5$, 
based on a code developed by F.-J. Sayas and his group at the University of Delaware.  A sufficiently refined grid is employed
to be able to focus on the temporal error. 
Instead of evaluating the  $H^{-1/2}$-error, we compute the quantity
\begin{align*}
  \max_{j=0,\dots n}\sqrt{\dualproduct{V(1)e_j}{e_j}} \qquad \text{ with } \qquad e_j:=\Pi_{L^2} \lambda(t_j) - \lambda^{k}(t_j).
\end{align*}
Here $\Pi_{L^2}$ denotes the $L^2$-orthogonal projection onto the BEM space. Since the grid is sufficiently fine and fixed,
this should not impact the observed convergence rates.
The operator $V(1)$ was taken because it gives an ($s$-independent) equivalent norm
  on $H^{-1/2}(\Gamma)$.

\begin{figure}
  \begin{subfigure}[b]{0.5\textwidth}
    \includeTikzOrEps{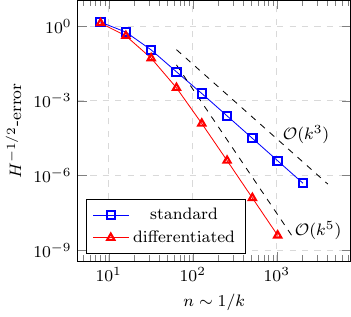}
    \caption{Comparison of the standard and differentiated method for 3-stage Radau IIA}
  \end{subfigure}
\hfill
  \begin{subfigure}[b]{0.5\textwidth}
    \includeTikzOrEps{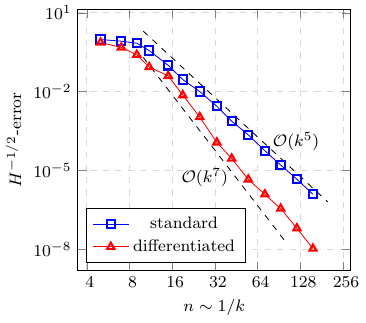}
    \caption{Comparison of the standard and differentiated method for 5-stage Radau IIA}
  \end{subfigure}
  \caption{Comparison of the standard and differentiated method for different RK schemes}
  \label{fig:comparison_of_methods}
\end{figure}

In Figure~\ref{fig:comparison_of_methods}, we observe that the rates from Proposition~\ref{prop:standard_method} and Theorem~\ref{thm:convergence_cq} are
obtained as predicted.
We conclude that while the fact that the rate jumps by order $2$, even though the modification of the scheme is of order one,
is  at first surprising, this can be rigorously explained by Theorem~\ref{thm:convergence_cq}. Observations of
this type provided the main motivation for the investigations in this work.

\textbf{Acknowledgments:} The authors gratefully acknowledge financial support by the Austrian Science Fund (FWF) through
the research program ``Taming complexity in partial differential systems'' (grant SFB F65).
\appendix
\section{Norm equivalence of interpolation spaces}
\begin{lemma}
\label{lemma:norm-equivalence}
Let ${\mathcal O}$ be a bounded domain. For $\rho > 0$ and $\theta \in (0,1)$ let $\|\cdot\|_{\rho,\theta,\mathcal{O}}$ be defined as in 
Definition~\ref{def:weighted_norms}. Define 
$$
\|u\|_{\theta,\mathcal{O}}:= \|u\|_{1,\theta,\mathcal{O}}, 
\qquad 
|u|_{\theta,\mathcal{O}}:= \|u - \overline{u}\|_{1,\theta,\mathcal{O}}, 
\qquad \overline{u}:= \frac{1}{|\mathcal{O}|} \int_{\mathcal{O}} u. 
$$
Then there are constants $c_1$, $c_2$ depending only on $\mathcal{O}$ and $\theta$ such that for all $u \in H^\theta(\mathcal{O})$ 
\begin{equation}
\label{eq:lemma:norm-equivalence}
c_1 \left( \rho^\theta\|u\|_{L^2(\mathcal{O})} + |u|_{\theta,\mathcal{O}}\right) \leq 
\|u \|_{\rho,\theta,\mathcal{O}} 
\leq c_2 \left( \rho^\theta\|u\|_{L^2(\mathcal{O})} + |u|_{\theta,\mathcal{O}}\right).
\end{equation}
\end{lemma}
\begin{proof}
We use \cite[Lemma~{4.1}]{karkulik-melenk-rieder20} with $X_0 = (L^2(\mathcal{O}), \|\cdot\|_{L^2(\mathcal{O})})$ 
and $X_1 = (H^1(\mathcal{O}), {\|\cdot\|_{1,\mathcal{O}}})$ there. As given there, we set 
$K(u,t):= \inf_{v \in H^1} {\|u - v\|_{L^2}} + t\|v\|_{H^1}$ as well as 
$k(u,t):= \inf_{v \in H^1} {\|u - v\|_{L^2}} + t|v|_{H^1}$, where we omitted the argument $\mathcal{O}$ for brevity. 
In \cite[Lemma~{4.1}]{karkulik-melenk-rieder20} the interpolation norm 
$\|\cdot\|_{\theta}$ based on $K$ and the interpolation seminorm $|\cdot|_{\theta}$ based on $k$. 
We note that $\|\cdot\|_\theta \sim \|\cdot\|_{\theta,\mathcal{O}}$. 

\emph{1.~step:} We claim $\|u - \overline{u}\|_{\theta} \sim |u|_\theta$. This claim follows from the following two estimates
using the Poincar\'e inequality: 
\begin{align*} 
K(u-\overline{u},t) & 
= \inf_{v \in H^1} \|u - \overline{u} - v\|_{L^2} + t\|v\|_{H^1} 
= \inf_{v \in H^1, c \in \R} \|u - \overline{u} - (v-c)\|_{L^2} + t\|v-c\|_{H^1}  \\
& \lesssim \inf_{v \in H^1} \|u - \overline{u} - (v-\overline{v})\|_{L^2} + t|v|_{H^1}  
= \inf_{v \in H^1, c \in \R} \|u - v - c\|_{L^2} + t|v|_{H^1}   \\
&=  \inf_{v \in H^1} \|u - v \|_{L^2} + t|v|_{H^1}  = k(u,t). 
\end{align*} 
Conversely, 
\begin{align*}
k(u,t) & = \inf_{v \in H^1,c\in \R} \|u - v - c\|_{L^2} + t |v|_{H^1}  
\leq  \inf_{v\in H^1} \|u - \overline{u} - v \|_{L^2} + t |v|_{H^1}   \\
& \leq  \inf_{v \in H^1} \|u - \overline{u} - v \|_{L^2} + t \|v\|_{H^1}  
 = K(u-\overline{u},t). 
\end{align*}
\emph{2.~step:} The norm equivalence \eqref{eq:lemma:norm-equivalence} follows from \cite[Lemma~{4.1}]{karkulik-melenk-rieder20}. 
\end{proof}
\bibliographystyle{alphaabbr}
\bibliography{literature}
\end{document}